\documentclass[parskip=full, titlepage=false, twoside=false]{scrartcl} 
\usepackage[utf8]{inputenc} 
\usepackage[T1]{fontenc} 
\usepackage[english]{babel}
\usepackage[autostyle=true]{csquotes}
\usepackage{lmodern} 
\usepackage[pdfauthor={Jonathan Glöckle}, pdftitle={Initial data rigidity via Dirac-Witten operators}]{hyperref} 
\usepackage{amsmath, amsthm, amssymb, bm, bbm, mathrsfs, mathtools}
\usepackage{extarrows}
\usepackage{tabulary,tabularx}
\usepackage{graphicx}
\usepackage{tikz-cd}
\usepackage{enumerate}
\usepackage[capitalise]{cleveref}
\usepackage{bbold}

\theoremstyle{plain} 
\newtheorem{Satz}{Theorem} 
\newtheorem{Prop}[Satz]{Proposition} 
\newtheorem{Lem}[Satz]{Lemma}
\newtheorem{Kor}[Satz]{Corollary} 
\theoremstyle{definition} 
\newtheorem{Def}[Satz]{Definition} 
\newtheorem{Bsp}[Satz]{Example}
\newtheorem{Bem}[Satz]{Remark} 

\newtheorem{Add}[Satz]{Addendum}
\newtheorem{Set}[Satz]{Setup}

\crefname{Satz}{Theorem}{Theorems}
\crefname{Prop}{Proposition}{Propositions}
\crefname{Lem}{Lemma}{Lemmata}
\crefname{Kor}{Corollary}{Corollaries}
\crefname{Def}{Definition}{Definitions}

\let\div\relax

\newcommand{\N}{\mathbbm{N}} 
\newcommand{\Z}{\mathbbm{Z}} 
 
\newcommand{\R}{\mathbbm{R}} 
\newcommand{\C}{\mathbbm{C}}

\newcommand{\del}{\partial}
\newcommand{\lto}{\longrightarrow}

\newcommand{\blank}{\,\cdot\,}
\newcommand{\dvol}{\mathrm{dvol}}
\newcommand{\DW}{\overline{D}}
\newcommand{\Hyp}{\overline{\Sigma}}
\newcommand{\bDW}{\overline{A}}
\newcommand{\bch}{\mathcal{X}}

\newcommand{\upd}{\mathrm{d}}
\newcommand{\dotcup}{\mathbin{\dot\cup}}

\renewcommand{\emptyset}{\varnothing}
\renewcommand{\epsilon}{\varepsilon}
\renewcommand{\phi}{\varphi}
\renewcommand{\hat}{\widehat}

\DeclareMathOperator{\div}{div}

\DeclareMathOperator{\spn}{span}
\DeclareMathOperator{\tr}{tr}
\DeclareMathOperator{\id}{\mathbb{1}}

\DeclareMathOperator{\End}{End}

\DeclareMathOperator{\Cl}{\C l}
\DeclareMathOperator{\Spin}{Spin}
\DeclareMathOperator{\SO}{SO}
\DeclareMathOperator{\scal}{scal}

\newlength{\keyheightlength}
\newlength{\meineboxheight}
\newcommand{\erhoehebox}[2]{\settoheight{\meineboxheight}{#2}\addtolength{\meineboxheight}{#1}\vbox
	to \meineboxheight{\vfil\hbox{#2}}}

\newcommand{\bafrac}[3][.06\keyheightlength]{\mathchoice%
	{\frac{\textstyle #2}{\erhoehebox{#1}{$\textstyle #3$}}}%
	{\frac{\scriptstyle #2}{\erhoehebox{#1}{$\scriptstyle #3$}}}%
	{\frac{\scriptscriptstyle #2}{\erhoehebox{#1}{$\scriptscriptstyle #3$}}}%
	{\frac{\scriptscriptstyle #2}{\erhoehebox{#1}{$\scriptscriptstyle #3$}}}
}

\begin{document}
\author{Jonathan Glöckle\thanks{Institutionen för Matematik, KTH Stockholm, Lindstedtsvägen 25, 100 44 Stockholm, Sweden, E-mail address: \href{mailto:jonathan.gloeckle.math@outlook.de}{jonathan.gloeckle.math@outlook.de}}}
\title{Initial data rigidity via Dirac-Witten operators}
\date{\today}
\maketitle

\begin{abstract}
	In this paper we prove an initial data rigidity result à la Eichmair, Galloway and Mendes \cite{Eichmair.Galloway.Mendes:2021} using Dirac operator techniques.
	It applies to initial data sets on spin bands that satisfy the dominant energy condition, a boundary condition for the future null expansion scalar and the $\hat{A}$-obstruction for positive scalar curvature on one of the boundary pieces.
	Interestingly, these bands turn out to carry lightlike imaginary $W$-Killing spinors, which are connected to Lorentzian special holonomy and moduli spaces of Ricci-flat metrics.
	We also obtain slight generalizations of known rigidity results on Riemannian bands.\\
	{\footnotesize 2020 \emph{Mathematics subject classification.} Primary: 53C21; Secondary: 53C27, 83C60.}
\end{abstract}

One of the main questions studied for positive scalar curvature is that of existence:
Given an $n$-manifold $M$, does it carry a metric of positive scalar curvature?
This question is answered by either providing a construction or finding a suitable obstruction.
Obstructionwise, two main answers have been found:
Firstly, Dirac operator techniques.
For instance, if $M$ is a closed spin manifold with non-vanishing $\hat{A}$-genus $\hat{A}(M)$, then with respect to any metric it carries a non-trivial Dirac-harmonic spinor and none of those can be of positive scalar curvature.
Secondly, minimal hypersurface techniques.
If~$M$ is a closed oriented manifold of dimension $2 \leq n \leq 7$ and assuming that there are cohomology classes $h_1, \ldots, h_{n-2} \in H^1(M,\Z)$ such that $[M] \cap (h_1 \cup \ldots \cup h_{n-2})$ is not in the image of the Hurewicz homomorphism $\pi_2(M) \to H_2(M, \Z)$, then $M$ does not carry a positive scalar curvature metric (cf.\ e.\,g.~\cite{Schick:1998}).
Actually, the upper dimension bound for the minimal hypersurface method can be improved to $11$.
This is a consequence of recent results by Chodosh, Mantoulidis, Schulze and Wang \cite{Chodosh.Mantoulidis.Schulze:2026,Chodosh.Mantoulidis.Schulze.Wang:2025p} building on earlier work by Nathan Smale \cite{Smale:1993} that allowed to raise the bound to $8$.

When trying to extend these obstruction results from positive to non-negative scalar curvature, we encounter rigidity phenomena.
Most prominently, if a closed manifold $M$ carries a non-negative scalar curvature metric $g$, but is known not to admit a positive scalar curvature metric, then $g$ must already be Ricci-flat (cf.\ e.\,g.~\cite[Lemma~5.2]{Kazdan.Warner:1975}).
Moreover, in the case where $M$ is spin with $\hat{A}(M) \neq 0$, the Riemannian manifold $(M,g)$ carries a non-trivial parallel spinor.
By studying closed Ricci-flat manifolds in more detail, further rigidity may be deduced.
For example, if additionally $b_1(M) \geq n$ (or $b_1(M) \geq n-1$ if $M$ is orientable), then by an argument of Bochner $(M,g)$ must even be isometric to a flat torus (cf.\ e.\,g.~\cite[Corollary~9.5.2]{Petersen:2016_3rd}).

When dealing with compact manifolds with boundary, even without closely looking at the geometry of Ricci-flat manifolds, there are surprisingly strong rigidity statements for non-negative scalar curvature metrics.
To obtain these, it is important to also assume appropriate boundary conditions.
Otherwise there are no obstructions to positive scalar curvature by Gromov's h-principle, and hence no rigidity in the sense of this article.
In~\cite{Baer.Hanke:2023} Bär and Hanke discuss and compare various boundary conditions.
A crucial role is played by the condition of mean convexity of the boundary.
This means that the mean curvature, defined as $H^g = \frac{1}{n-1}\tr(-\nabla \nu)$ for the inward-pointing unit normal $\nu$, is non-negative. 
Among other things, they show the following rigidity statement \cite[Thm.~19]{Baer.Hanke:2023}:
If $M$ is a compact connected spin manifold with boundary that has stably infinite $K$-area, then any Riemannian metric $g$ with $\scal^g \geq 0$ and $H^g \geq 0$ is Ricci-flat with $H^g \equiv 0$.
In particular, positive scalar curvature metrics with mean convex boundary are obstructed on such manifolds.
The theorem remains true for manifolds $M^\prime \times N$, where $M^\prime$ has stably infinite $K$-area as above and $N$ is a closed spin manifold with non-zero $\hat{A}$-genus; in particular, it applies to $M=[0,1] \times N$.

In this article, we will obtain a strengthening of this special case, where we do not a priori assume a cylindrical form, but just suppose that $M$ is what is sometimes called a \emph{band}.
This means that the boundary is decomposed into two pieces $\del M = \del_+ M \dotcup \del_- M$  as a topological disjoint union, so $\del_+ M$ and $\del_- M$ are unions of components of $\del M$.
Also, we obtain a more explicit description of the metrics in the rigidity case.
 
\begin{Kor} \label{Cor:Riem1}
	Let $(M,g)$ be a compact connected Riemannian spin manifold with boundary $\del M = \del_+ M \dotcup \del_- M$.
	Assume that
	\begin{itemize}
		\item $g$ has non-negative scalar curvature $\scal^g \geq 0$,
		\item the boundary is mean convex, i.\,e.~$H^g \geq 0$ with respect to the inward-pointing unit normal on $\del M$, and
		\item the $\hat{A}$-genus of $\del_- M$ is non-zero: $\hat{A}(\del_- M) \neq 0$.
	\end{itemize}
	Then $(M,g)$ is isometric to $(\del_-M \times [0,\ell], \gamma + \upd t^2)$ for a Ricci-flat metric $\gamma$ on $\del_- M$ admitting a non-trivial parallel spinor.
\end{Kor}

A very similar statement was shown by Räde \cite[Thm.~2.14]{Raede:2023}.
The major difference is that he assumes a dimension bound and that any closed embedded hypersurface between $\del_- M$ and $\del_+ M$ does not admit positive scalar curvature so that minimal hypersurface (more precisely: $\mu$-bubble) techniques can be applied.
On the other hand he has no need of the spin and the $\hat{A}$-condition on $\del_- M$ that we use for our spinorial proof.

In our case, \Cref{Cor:Riem1} arises as a byproduct of studying rigidity for initial data sets.
By an \emph{initial data set} on a manifold $M$ we understand a pair $(g, k)$ consisting of a Riemannian metric $g$ and a symmetric $2$-tensor field $k$ on $M$.
They naturally appear in the following way:
If $M$ is a spacelike hypersurface in a time-oriented Lorentzian manifold $(\overline{M}, \overline{g})$, then there is an induced initial data set $(g, k)$ on $M$, where $g$ is the induced Riemannian metric on $M$ and $k$ is its second fundamental form with respect to the future-pointing unit normal $e_0$ (cf.~\eqref{eq:2FF}).
They serve as initial data for the Cauchy problem of general relativity, together with initial data for the matter fields under consideration.
In particular, the pair $(g, k)$ is all the initial data needed for the Cauchy problem in the vacuum case, where \emph{energy density} $\rho \coloneqq \frac{1}{2}(\scal^g + \tr^g(k)^2 - |k|_g^2)$ and \emph{momentum density} $j \coloneqq \div^g(k) - \upd \tr^g(k)$ vanish identically.
More generally, the initial data sets of physical interest are the ones that satisfy the \emph{dominant energy condition} $\rho \geq |j|_g$.
Note that if $k \equiv 0$, then $(g, k)$ satisfies the dominant energy condition if and only if $\scal^g \geq 0$.

Let now $(g,k)$ be an initial data set on a manifold $M$, potentially with boundary $\del M$.
We consider a co-oriented hypersurface $F$ in $M$.
The co-orientation will be given by a unit normal vector field $\tilde{\nu}$.
Following physics literature, we refer to the direction of $\tilde{\nu}$ as \emph{outgoing} and make the following definitions:
The \emph{future outgoing null second fundamental form} $\chi^+ \in \Gamma(\bigodot^2 T^*F)$ is defined by $\chi^+ = g(\nabla \tilde{\nu}, -) + k_{|F}$. 
Its trace is the \emph{future outgoing null expansion scalar} $\theta^+ = \tr^F(\chi^+) =  \tr^F(\nabla \tilde{\nu}) + \tr^F(k)$.

Geometrically, its significance is the following.
If $(g, k)$ is the induced initial data set on a hypersurface $M$ of a time-oriented Lorentzian manifold $(\overline{M}, \overline{g})$, then the second fundamental form of $F$ in $\overline{M}$ can be expressed in the normal frame of $F$ given by the null vector fields $\tilde{\nu} + e_0$ and $\tilde{\nu} - e_0$.
In this case the coefficient in front of $\tilde{\nu} - e_0$ is given by $-\frac{1}{2} \chi^+$.
Thus for any compactly supported variation  $(F_t)_{t \in (-\epsilon, \epsilon)}$ of $F = F_0$ in $\overline{M}$ with variation vector field $f \cdot (\tilde{\nu} + e_0)$, $f \in C^\infty_c(F)$, the first variation of the volume is given in terms of the null expansion scalar, namely it is equal to $\int_F f\div^F(\tilde{\nu} + e_0) \dvol = \int_F f\theta^+ \dvol$.

If $\theta^+ \equiv 0$, then $F$ is called a \emph{MOTS} (which stands for \emph{marginally outer trapped surface}).
Note that if $k \equiv 0$, then $\chi^+$ and $\theta^+$ reduce to the second fundamental form of $F$ in $M$ and (a multiple of) its mean curvature, respectively.
A MOTS is then just a minimal surface.

With these notions at hand, we can formulate our main theorem.
The role of $F$ will be played at first by the boundary pieces $\del_+ M$ and $\del_- M$, later by the leaves $F_t$ of a foliation extending $F_0 = \del_+ M$ and $F_\ell = \del_- M$.
Notice that, somewhat confusingly, on $\del_+ M$ the outgoing unit normal $\tilde \nu$ is chosen to be the inward-pointing one.

\begin{Satz} \label{Thm:Main}
	Let $M$ be a compact connected spin manifold with boundary $\del M = \del_+ M \dotcup \del_- M$ endowed with an initial data set $(g,k)$.
	Denote by $\tilde{\nu}$ the unit normal on $\del M$ that is inward-pointing along $\del_+ M$ and outward-pointing along $\del_- M$.
	Assume that
	\begin{itemize}
		\item $(g,k)$ satisfies the dominant energy condition $\rho \geq |j|_g$,
		\item the future null expansion scalar (with respect to $\tilde{\nu}$) satisfies $\theta^+ \leq 0$ on $\del_+ M$ and $\theta^+ \geq 0$ on $\del_- M$, and
		\item the $\hat{A}$-genus of $\del_- M$ is non-zero: $\hat{A}(\del_- M) \neq 0$.
	\end{itemize}
	Then there is a diffeomorphism $\Phi \colon \del_- M \times [0,\ell] \to M$ defining a foliation $F_t = \Phi(\del_- M \times \{t\})$ with $F_0 = \del_+ M$ and $F_\ell = \del_- M$.
	The leaves can be endowed with an induced initial data set, an induced spin structure and a unit normal $\tilde{\nu}$ pointing in the direction of growing $t$-parameter.
	The diffeomorphism can be chosen in such a way that the following holds for every leaf $F_t$:
	\begin{itemize}
		\item Its future null second fundamental form (with respect to $\tilde{\nu}$) vanishes, $\chi^+ = 0$, in particular it is a MOTS.
		\item It carries a non-trivial parallel spinor, in particular its metric is Ricci-flat.
		\item Its tangent vectors are orthogonal to $j^\sharp$ and $\rho + j(\tilde{\nu}) = 0$, in particular the dominant energy condition holds marginally: $\rho = |j|_g$.
	\end{itemize}
\end{Satz}

Let us remark that $\hat{A}(\del_- M) \neq 0$ in particular implies that $\del_- M$ is non-empty.
The same is true for $\del_+ M$ since $\del_+ M$ is spin bordant to $\del_- M$ and thus $|\hat{A}(\del_+ M)| = |\hat{A}(\del_- M)|$.
We also see that the theorem is symmetric under exchanging $\del_+ M$ with $\del_- M$ (i.\,e.\ flipping the orientation of $\tilde \nu$) if at the same time $k$ is replaced by $-k$.

Initial data rigidity was first studied by Eichmair, Galloway and Mendes in their recent paper \cite{Eichmair.Galloway.Mendes:2021}.
They did so using minimal hypersurface (or rather: MOTS) techniques, whereas this article follows a spinorial approach to the problem.
Let us compare \cref{Thm:Main} with their result \cite[Thm.~1.2]{Eichmair.Galloway.Mendes:2021} in little more detail.
The main setup is the same:
They also consider initial data sets on compact connected manifolds with boundary satisfying the dominant energy condition and the boundary condition for $\theta^+$.
Then, there is an assumption that excludes positive scalar curvature on one of the boundary pieces.
In our case, this is provided by $\hat{A}(\del_- M) \neq 0$ and the observation that $\del_- M$ is spin.
In their case, it is what they call cohomology condition -- existence of classes $h_1, \ldots h_{n-1} \in H^1(\del_- M, \Z)$ with $h_1 \cup \ldots \cup h_{n-1} \neq 0$ -- together with the dimension bound $2 \leq n-1 \leq 6$.
As a last assumption some “weak niceness” of the boundary inclusion $\del_- M \hookrightarrow M$ is needed.
We need a spin structure of $\del_- M$ to extend to $M$; they require the so-called homotopy condition, i.\,e.\ that there is a continuous map $M \to \del_- M$ so that the composition $\del_- M \hookrightarrow M \to \del_- M$ is homotopic to the identity. 
The conclusions almost coincide:
Both theorems show $M \cong \del_- M \times [0, \ell]$ such that the canonical leaves are Ricci-flat manifolds with $\chi^+ = 0$ and $j^\sharp = -\rho \tilde{\nu}$.
In our theorem, we additionally obtain existence of a non-trivial parallel spinor on the leaves -- a feature we are going to discuss in more detail below.
Since Eichmair, Galloway and Mendes impose strong enough conditions on $\del_- M$ to make use of the argument by Bochner mentioned above, they are able to further conclude that the leaves are isometric to flat tori.

In a new article~\cite{Galloway.Mendes:2024}, Galloway and Mendes also discuss initial data rigidity for closed manifolds.
An analogue of \cite[Thm.~4.3]{Galloway.Mendes:2024} is the following immediate consequence of \cref{Thm:Main}:
If $(g,k)$ is an initial data set with $\rho \geq |j|_g$ on a closed connected spin manifold $M$ that contains a MOTS $F$ with $\hat{A}(F) \neq 0$, then $M$ is diffeomorphic to a mapping torus $F \times [0, \ell]/(x,\ell) \sim (f(x),0)$, $f \in \mathrm{Diff}(F)$, in such a way that the canonical leaves satisfy all the properties listed in \cref{Thm:Main}.

The author's interest in initial data rigidity arose from studying the space of initial data sets subject to the dominant energy condition on a fixed manifold $M$.
In \cite{Gloeckle:2023} and \cite{Gloeckle:2024}, it was shown that for many choices of $M$ this space has non-trivial homotopy groups and, more importantly, different connected components -- if the strict version $\rho > |j|_g$ of the dominant energy condition is considered.
For statements of increased physical relevance, this condition should be relaxed to the non-strict dominant energy condition $\rho \geq |j|_g$. 
For instance, in \cite{Ammann.Gloeckle:2023} Bernd Ammann and the author discuss that for certain manifolds $M$ any spacetime $(\overline{M}, \overline{g})$ containing $M$ as Cauchy hypersurface and satisfying the spacetime dominant energy condition cannot have both a big bang and a big crunch singularity.
The main step there is the mentioned passage from strict to non-strict inequality.
This is done by examining how rigid the equality case is.

More precisely, in the situation of \cite{Ammann.Gloeckle:2023}, a bit more is known about the equality case of interest:
There exists a spinor $\phi \not\equiv 0$ that is parallel with respect to the connection $\overline{\nabla}_X \phi = \nabla_X \phi + \frac{1}{2} k(X,-)^\sharp \cdot e_0 \cdot \phi$.
Here, $\phi$ is a section of the \emph{hypersurface spinor bundle} $\Hyp M \to M$ (cf.~\cref{Sec:HypSpin}), $e_0 \cdot \colon \Hyp M \to \Hyp M$ is a Clifford-antilinear involution it comes equipped with and $\nabla$ is induced by the Levi-Civita connection of $(M,g)$.
These $\overline{\nabla}$-parallel spinors  (which are also known under the name \emph{imaginary $W$-Killing spinor}) come in two flavors, depending on whether their Lorentzian Dirac current $V_\phi = u_\phi e_0 - U_\phi \in \Gamma(TM \oplus \underline{\R}e_0)$  (cf.~\cref{Def:DirCurr}) is timelike or lightlike.
Especially the lightlike ones have attracted attention, since they play an important role in the study of Lorentzian special holonomy \cite{Baum.Leistner.Lischewski:2016}.

In the proof of \cref{Thm:Main}, one main step will be to show existence of a non-trivial $\overline{\nabla}$-parallel spinor.
As it turns out, non-emptiness of the boundary helps since a boundary condition forces the spinor to be lightlike.
From there, the other conclusions will be deduced by considering the foliation defined by $U_\phi$.
Since this intermediate result might be of interest in the future, we formulate it more explicitly.

\begin{Add} \label{Add:Main}
	Under the assumptions of \cref{Thm:Main}, the initial data set $(g,k)$ on $M$ carries a lightlike $\overline{\nabla}$-parallel spinor $\phi$.
	The foliation $(F_t)_{t \in [0,\ell]}$ in \cref{Thm:Main} may be constructed in such a way that the Riemannian Dirac current $U_\phi$ of $\phi$ is orthogonal\footnote{Actually, pointing in the direction of $-\tilde{\nu}$.} to the leaves $F_t$.
\end{Add}

One might ask whether even more rigidity can be deduced, similarly as in \cref{Cor:Riem1} where the metric stays the same on all leaves.
Nowadays, it is rather well understood that this is not possible.
In fact, Bernd Ammann, Klaus Kröncke and Olaf Müller gave a method for constructing lightlike $\overline{\nabla}$-parallel spinors on cylinders $N \times [0,L]$ in \cite{Ammann.Kroencke.Mueller:2021} providing many examples of initial data sets as in \cref{Thm:Main}.
Namely, from a Ricci-flat metric $\gamma_0$ on $N$ with a parallel spinor $\phi_0 \not\equiv 0$, a smooth curve $([\tilde\gamma_t])_{t \in [0,L]}$ in the moduli space of Ricci-flat metrics on $N$ starting at $[\tilde\gamma_0] = [\gamma_0]$ and a smooth function $f \colon [0,L] \to \R$ they construct an initial data set $(g, k) = (\gamma_t + \upd t^2, \frac{1}{2}\frac{\del}{\del t}\gamma_t + f(t) \upd t^2)$ and a lightlike $\overline{\nabla}$-parallel spinor $\phi$ on $N \times [0,L]$ such that $[\gamma_t] = [\tilde\gamma_t]$ for all $t \in [0,L]$, $\phi_{|N \times \{0\}} = \phi_0$ and $U_\phi$ is orthogonal to the canonical leaves.
While all these cylinders feature that $|\phi|^2 = |U_\phi|_g$ is constant along the leaves (this norm is given in terms of $|\phi_0|$ and $f$), the construction was recently improved by Bernd Ammann and Klaus Kröncke together with the author \cite{Ammann.Gloeckle.Kroencke:2026p} to get rid of this restriction on the norm. 
This leads to a complete classification of the initial data sets appearing in \cref{Thm:Main}.

We conclude our discussion with a rather general rigidity statement for Riemannian bands, which essentially follows from \cref{Thm:Main}.
Though its assumptions might seem rather technical, they nicely fit into the context of warped products.

\begin{Bsp}
	Consider a warped product $(\tilde{M}, \tilde{g}) = (N \times [0,L], w(s)^2 \tilde{\gamma} + \upd s^2)$ for a Riemannian manifold $(N, \tilde{\gamma})$ and a warping function $w \colon [0,L] \to \R$.
	Setting $h = \frac{w^\prime}{w}$, its scalar curvature is given by $\scal^{\tilde{g}} = (w \circ s)^{-2}\scal^{\tilde{\gamma}}-n(n-1)(h \circ s)^2 - 2(n-1) h^\prime \circ s$, where $s$ denotes the canonical projection on the $[0,L]$-factor.
	Moreover, the mean curvature of the leaf $N \times \{s\}$ with respect to the unit normal $\frac{\del}{\del s}$ is given by $-h(s)$.
	This means that the mean curvature of the boundary with respect to the inward-pointing unit normal is $H^{\tilde{g}} = -h(0)$ along $\del_+ \tilde{M} = N \times \{0\}$ and $H^{\tilde{g}} = h(L)$ along $\del_- \tilde{M} = N \times \{L\}$.
	Let us now assume that the scalar curvature of $(N, \tilde{\gamma})$ is non-negative, the warping function is log-concave, i.\,e.\ $\frac{\upd^2}{\upd s^2} \log(w) = h^\prime \leq 0$, and there exists a  $1$-Lipschitz map $\tilde{\Phi} \colon (M, g) \to (\tilde{M}, \tilde{g})$ sending $\del_+ M$ to $N \times \{0\}$ and $\del_- M$ to $N \times \{L\}$ such that $\scal^g \geq \scal^{\tilde{g}} \circ \tilde{\Phi}$ (on $M$) and $H^g \geq H^{\tilde{g}} \circ \tilde{\Phi}$ (along $\del M$).
	Then $h$ and $s \circ \tilde{\Phi} \colon M \to [0,L]$ satisfy the assumptions of \cref{Thm:Riem}.
\end{Bsp}

With this example in mind, \cref{Thm:Riem} may thus be read as a comparison result.

\begin{Satz} \label{Thm:Riem}
	Let $(M,g)$ be a compact connected Riemannian spin manifold of dimension~$n$ with boundary $\del M = \del_+ M \dotcup \del_- M$.
	Suppose $h \colon [0,L] \to \R$ is a smooth function with $h^\prime \leq 0$ and $s \colon M \to [0,L]$ is a smooth map with $s(\del_+ M) = \{0\}$ and $s(\del_- M) = \{L\}$ and such that $|\upd s|_g \leq 1$.
	Assume that
	\begin{itemize}
		\item the scalar curvature of $g$ is bounded below by $\scal^g \geq -n(n-1) (h \circ s)^2 - 2(n-1) h^\prime \circ s$,
		\item the mean curvature of the boundary with respect to the inward-pointing unit normal $\nu$ is bounded below by $H^g \geq -h(0)$ on $\del_+ M$ and $H^g \geq h(L)$ on $\del_- M$, and
		\item the $\hat{A}$-genus of $\del_- M$ is non-zero: $\hat{A}(\del_- M) \neq 0$.
	\end{itemize}
	Then there is an isometry $\Phi \colon (\del_-M \times [0,\ell], v(t)^2\gamma + \upd t^2) \to (M,g)$ with $\Phi(\del_- M \times \{0\}) = \del_+ M$ and $\Phi(\del_- M \times \{\ell\}) = \del_- M$, where $v \colon [0,\ell] \to \R$ is a smooth function and $\gamma$ is a Ricci-flat metric on $\del_- M$ admitting a non-trivial parallel spinor.
	More precisely, the composition $h \circ s$ is constant along the leaves of the canonical foliation and -- reinterpreting $h \circ s$ as function $[0,\ell] \to \R$ -- the warping function $v$ is determined (up to multiplication by a constant) by $\frac{v^\prime}{v} = h \circ s$.
	Moreover, $\upd s = (\Phi^{-1})^* \upd t$ wherever $h^\prime \circ s \neq 0$ and the inequalities for $scal^g$ and $H^g$ are equalities.
\end{Satz}

Again, there is a symmetry interchanging $\del_+ M$ with $\del_- M$.
This involves replacing $s$ by $\sigma \circ s$ and $h$ by $-h \circ \sigma$, where $\sigma \colon [0,L] \to [0,L]$ is the affine linear map switching the boundaries.
Furthermore, the assumption on $s$ can be slightly weakened.
\begin{Bem}
	\Cref{Thm:Riem} still holds true when the condition $|\upd s|_g \leq 1$ is only satisfied on the subset of $M$ where $h^\prime \circ s \neq 0$.
\end{Bem}

Comparable statements were already derived in articles by Cecchini and Zeidler \cite{Cecchini.Zeidler:2024} and Räde \cite{Raede:2023}.
The theorem of Cecchini and Zeidler \cite[Thm.~8.3 (cf.~also Thms.~9.1 and 10.2)]{Cecchini.Zeidler:2024} is also derived using spinor techniques.
It is more general in the sense that it also allows for non-trivial twist bundles $E \to M$ with the consequence that the index obstruction can be relaxed to $0 \neq \hat{A}(\del_- M, E) = \int_{\del_- M} \hat{A}(T\del_- M) \wedge \mathrm{ch}(E_{|\del_- M})$.
On the other hand, it is more restrictive as it requires the strict inequality $h^\prime < 0$.
In this case, the \emph{band width} $\mathrm{dist}^g(\del_+ M, \del_- M)$ plays a crucial role as the function $s$ needs to be $1$-Lipschitz.
Since a priori there does not need to exist a smooth $1$-Lipschitz function $s$ realizing the width, meaning $L = \mathrm{dist}^g(\del_+ M, \del_- M)$, it is also interesting to allow for non-smooth Lipschitz functions in the case $h^\prime < 0$.
Cecchini and Zeidler's article contains some arguments for this.

As already mentioned above, in Räde's work the $\hat{A}$- and the spin condition are replaced by conditions needed for a $\mu$-bubble argument to work.
His main theorem captures both the case $h^\prime < 0$ and the case $h^\prime \equiv 0$, but only in the latter case he is able to derive a rigidity statement comparable to \cref{Thm:Riem}.

Although the general case $h^\prime \leq 0$ seems to be new, the main applications of \cref{Thm:Riem} are the ones, where $h$ is such that the lower scalar curvature bound is a constant.
Then $h$ is subject to an ordinary differential equation and either $h^\prime < 0$ or $h^\prime \equiv 0$.
These functions and the associated corollaries are discussed in \cite[Sec.~2.A]{Raede:2023}.
We restrict our attention just to the cases $h \equiv 0$, which yields \cref{Cor:Riem1}, and $h \equiv -1$ yielding \cref{Cor:Riem2} below, where our theorem supersedes the result of Cecchini and Zeidler.
(Setting $h \equiv 1$ yields the statement of \cref{Cor:Riem2} with interchanged boundary pieces.)
Notice that a statement analogous to \cref{Cor:Riem1,Cor:Riem2} is also contained in the article by Eichmair, Galloway and Mendes \cite[Cor.~1.4]{Eichmair.Galloway.Mendes:2021}.

\begin{Kor} \label{Cor:Riem2}
	Let $(M,g)$ be a compact connected Riemannian spin manifold with boundary $\del M = \del_+ M \dotcup \del_- M$ of dimension $n$.
	Assume that
	\begin{itemize}
		\item the scalar curvature is bounded below by $\scal^g \geq -n(n-1)$,
		\item the mean curvature of the boundary with respect to the inward-pointing unit normal is bounded below by $H^g \geq 1$ on $\del_+ M$ and $H^g \geq -1$ on $\del_- M$, and
		\item the $\hat{A}$-genus of $\del_- M$ is non-zero: $\hat{A}(\del_- M) \neq 0$.
	\end{itemize}
	Then $(M,g)$ is isometric to $(\del_-M \times [0,\ell], e^{-2t}\gamma + \upd t^2)$ (with $\del_+ M$ corresponding to $\del_- M \times \{0\}$) for a Ricci-flat metric $\gamma$ on $\del_- M$ admitting a non-trivial parallel spinor.
\end{Kor}

Let us finally discuss the strategy of the proof of \cref{Thm:Main} and the structure of this article.
The main tool we are going to use is the \emph{Dirac-Witten operator} $\DW$, which lives on the hypersurface spinor bundle $\Hyp M \to M$ mentioned above.
This bundle is best explained if $M$ is assumed to sit as a spacelike hypersurface in a time-oriented Lorentzian spin manifold $(\overline{M}, \overline{g})$.
In this case, $\Hyp M \to M$ is just the restriction to $M$ of a spinor bundle on $\overline{M}$.
In particular, it carries an involution $e_0 \cdot$ induced by Clifford multiplication with the future unit normal on $M$ and a connection $\overline{\nabla}$ induced by the Levi-Civita connection of $(\overline{M}, \overline{g})$.
In \cref{Sec:HypSpin}, we discuss how to construct this bundle and its additional structures intrinsically, from the initial data set $(g, k)$ and a spin structure on $M$ alone.
The Dirac-Witten operator $\DW$ is the Dirac operator of $\Hyp M$ with respect to its connection $\overline{\nabla}$.
The hypersurface spinor bundle also carries a connection $\nabla$ induced from the Levi-Civita connection of $(M,g)$ and there is a Dirac operator $D$ associated to it.
As it turns out, they are related via $\DW = D - \frac{1}{2} \tr^g(k) e_0 \cdot$.
This means that $\DW$ is a Callias operator in the sense of Cecchini and Zeidler.
Putting \emph{chirality boundary conditions} $\tilde{\nu} \cdot \phi = -e_0 \cdot \phi$ on the sections of $\Hyp M \to M$, we are able to invoke their analytical results.
This is done in \cref{Sec:DWCallias}.
As a result, we obtain existence of non-trivial Dirac-Witten harmonic spinors $\phi$ subject to chirality boundary conditions if $\hat{A}(\del_- M) \neq 0$.
These spinors are then further studied in \cref{Sec:KerDW} using an integrated version of the Schrödinger-Lichnerowicz formula $\DW^2 = \overline{\nabla}^*\overline{\nabla} + \frac{1}{2}(\rho - e_0 \cdot j^\sharp \cdot)$.
If the dominant energy condition holds and the inequalities for $\theta^+$ are satisfied along $\del M$, we are able to conclude that $\phi$ is a lightlike $\overline{\nabla}$-parallel spinor.
We continue by studying the foliation defined by its Riemannian Dirac current $U_\phi$.
Doing so, we prove the main theorems -- up to the observation that this foliation is actually of cylindrical type $\del_- M \times [0, \ell]$.
The remaining piece is provided in \cref{Sec:ProdStr}, where we look at the flow of $-\bafrac{U_\phi}{|U_\phi|_g^2}$ in more detail.

\subsection*{Acknowledgments}
I am grateful to Bernd Ammann for pointing me to this problem and for all his support and advice.
Thanks to the anonymous referees for their thorough reviews.
I also like to thank Thomas Schick for his careful review of this article as part of his report on my PhD thesis.
During the execution of this project, I was supported by the SFB 1085 “Higher Invariants” funded by the DFG.

\section{Spinor bundles on hypersurfaces} \label{Sec:HypSpin}
This section is devoted to the study of spinor bundles on hypersurfaces.
This is to be understood in a two-fold manner:
First, we are interested in the situation where $M$ is an $n$-dimensional spacelike hypersurface of a time-oriented Lorentzian manifold $\overline{M}$.
Second, we assume that $M$ has boundary $\del M$ and restrict the spinor bundle further to $\del M$.

The first step is known under the name hypersurface spinor bundle, cf.~\cite{Hijazi.Zhang:2003,Ammann.Gloeckle:2023}.
The construction is the following:
Given a complex, say, representation $\Cl_{n,1} \to \End(W)$ and a spin structure $P_{\Spin(n)}M \to M$ of the Riemannian spin $n$-manifold $(M,g)$, we form the \emph{hypersurface spinor bundle} $\Hyp M \to M$ by associating $W$ to $P_{\Spin(n)}M$ via the restricted representation $\Spin(n) \hookrightarrow \Cl_n \hookrightarrow \Cl_{n,1} \to \End(W)$.
To justify the name, we assume that $(M,g)$ is a spacelike hypersurface (with induced metric) of a space- and time-oriented Lorentzian manifold $(\overline{M},\overline{g})$.
We moreover assume that $\overline{M}$ is spin (which can be assured by restricting to a small neighborhood of $M$) and the spin structure $P_{\Spin_0(n,1)}\overline{M} \to \overline{M}$ restricts to the one of $M$ in the sense that
\begin{equation} \label{eq:spinpb}
	\begin{tikzcd}
		P_{\Spin(n)} M \rar \dar & P_{\Spin_0(n,1)} \overline{M}_{|M} \dar \\
		P_{\SO(n)} M \rar & P_{\SO_0(n,1)} \overline{M}_{|M} \\[-1.8em]
		(e_1,\ldots,e_n) \rar[mapsto] & (e_0,e_1,\ldots,e_n)
	\end{tikzcd}
\end{equation}
is a pullback, where $e_0$ denotes the future unit normal of $M$ in $(\overline{M},\overline{g})$.
Then the spinor bundle $\Hyp \overline{M} \to \overline{M}$ associated to the representation $\Cl_{n,1} \to \End(W)$ restricts to the hypersurface spinor bundle on $M$, meaning that the canonical map yields a bundle isomorphism $\Hyp M \cong \Hyp \overline{M}_ {|M}$.

The hypersurface spinor bundle can be equipped with additional structures.
First of all, it comes with a Clifford multiplication $T\overline{M}_{|M} \otimes \Hyp M \to \Hyp M$.
For our purposes, it is more convenient to view it as a Clifford multiplication by vectors of $TM$ and an involution $e_0 \cdot \colon \Hyp M \to \Hyp M$ that anti-commutes with the $TM$-Clifford multiplication.
Secondly, if $W$ admits a $\Z/2\Z$-grading and the representation $\Cl_{n,1} \to \End(W)$ is a graded representation, then $\Hyp M$ carries a $\Z/2\Z$-grading with respect to which the Clifford multiplication and the involution $e_0 \cdot$ are odd.
Thirdly, $W$ can be endowed with a (positive definite) scalar product that is invariant under multiplication by the standard basis vectors $E_0, E_1, \ldots, E_n \in \R^{n,1} \subseteq \Cl_{n,1}$.
Such a scalar product can always be constructed by an averaging procedure.
It can also be made compatible with the $\Z/2\Z$-grading if $W$ carries one.
Such a scalar product on $W$ induces a fiberwise scalar product on $\Hyp M$ such that the Clifford multiplication by vectors in $TM$ is skew-adjoint, the involution $e_0 \cdot$ is self-adjoint and the $\Z/2\Z$-grading is orthogonal.

The last structure we want to consider is the one of a connection.
There are two canonical choices.
The Levi-Civita connection of $(M,g)$ gives rise to a connection $\nabla$ on $P_{\Spin(n)}M$ and hence on $\Hyp M$.
On the other hand, the Levi-Civita connection of $(\overline M, \overline g)$ induces a connection $\overline{\nabla}$ on $\Hyp \overline{M}$ and thus also on $\Hyp M$.
On tangent bundles the Levi-Civita connections of $(\overline{M}, \overline{g})$ and the hypersurface $M$ differ by the second fundamental form $k$:
\begin{equation} \label{eq:2FF}
	\overline{\nabla}_X Y = \nabla_X Y + k(X,Y)\,e_0
\end{equation}
for all $X,Y \in \Gamma(TM)$.
It follows that a similar relation also holds for the connections on $\Hyp M$:
\begin{equation} \label{eq:olnabla}
	\overline{\nabla}_X \phi = \nabla_X \phi + \frac12 k(X,-)^\sharp \cdot e_0 \cdot \phi
\end{equation}
for all $X \in \Gamma(TM)$ and $\phi \in \Gamma(\Hyp M)$.
This formula allows to define the connection $\overline{\nabla}$ even in the case when $M$ is not embedded as a hypersurface.
It is only necessary to have a metric $g$ and a symmetric $2$-tensor $k$ playing the role of the second fundamental form.
Pairs $(g,k)$ of this kind are known as \emph{initial data sets} on $M$.
From the way it is defined, it is clear that grading, Clifford multiplication, the involution $e_0 \cdot$ and scalar product are parallel with respect to $\nabla$.
Compatibility formulae for $\overline{\nabla}$ may be derived using \eqref{eq:olnabla}.

\begin{Set} \label{Set:HypSpin}
	Given an initial data set $(g,k)$ on a spin manifold $M$, we form a hypersurface spinor bundle $\Hyp M \to M$ with the structures of a $TM$-Clifford multiplication, an involution $e_0 \cdot$, a (positive definite) scalar product and a connection $\overline{\nabla}$.
	They satisfy the compatibility conditions described in the previous two paragraphs.
	When forming a $\Z/2\Z$-graded hypersurface spinor bundle, we also require the grading to be compatible with the other structures in the above-described sense.
\end{Set}

For step two, let $M$ furthermore have boundary $\del M$.
The inward-pointing unit normal along $\del M$ will be denoted by $\nu$.
The hypersurface spinor bundle restricts to $\Hyp M_{|\del M} \to \del M$, to which we refer as \emph{boundary hypersurface spinor bundle}.
It may, similarly as explained above, also be defined on $\del M$ intrinsically.
From that perspective the $TM_{|\del M}$-Clifford multiplication can be seen as a $T(\del M)$-Clifford multiplication together with a homomorphism $\nu \cdot$ anti-commuting with this Clifford multiplication and squaring to~$-\id$.

The boundary hypersurface spinor bundle carries even more connections of interest.
Of course, the connections $\overline{\nabla}$ and $\nabla$ restrict to $\Hyp M_{|\del M} \to \del M$.
Viewing the boundary hypersurface spinor bundle as bundle associated to the induced spin structure on $\del M$, we obtain the connection $\nabla^\del$ induced by the Levi-Civita connection of $\del M$.
We have
\begin{equation} \label{eq:SpBConn}
	\nabla_X \phi = \nabla^\del_X \phi - \frac12 (\nabla_X \nu) \cdot \nu \cdot \phi
\end{equation}
for all $X \in \Gamma(T(\del M))$ and $\phi \in \Gamma(\Hyp M_{|\del M})$.

There is another, less obvious choice.
For this, we observe that every metric connection on $T\overline{M}_{|\del M}$ gives rise to a connection on $P_{\Spin_0(n,1)} \overline{M}_{|\del M}$ and thus on the boundary hypersurface spinor bundle.
In this way the Levi-Civita connection of $(\overline{M}, \overline{g})$ induces $\overline{\nabla}$.
Equipping $T\overline{M}_{|\del M} = TM_{|\del M} \oplus \underline{\R}e_0$ with sum of the Levi-Civita connection of $(M,g)$ and the trivial connection we obtain $\nabla$.
The connection $\nabla^\del$ arises when we put on $T\overline{M}_{|\del M} = T(\del M) \oplus \underline{\R}\nu \oplus \underline{\R}e_0$ the sum of the Levi-Civita connection of $\del M$ with the trivial connection on the other summands.
Now, instead, let us take on the normal bundle $N(\del M) = \underline{\R}\nu \oplus \underline{\R}e_0$ the connection induced by $\overline{\nabla}$, i.\,e.\ $\pi^{\mathrm{nor}}(\overline{\nabla}_X n)$ with $X \in \Gamma(T(\del M))$, $n \in \Gamma(N(\del M))$ and $\pi^{\mathrm{nor}} \colon T\overline{M}_{|\del M} \to N(\del M)$ the orthogonal projection.
We obtain a connection on the boundary hypersurface spinor bundle, which we denote by $\overline{\nabla}^\del$.

As before, there is a simple comparison formula with (one of) the other connections.
This time, we provide a proof, which should also serve as a blueprint for the other claims made.
\begin{Lem} \label{Lem:BConn}
	The connection $\overline{\nabla}^\del$ satisfies
	\begin{align*}
		\overline{\nabla}_X \phi &= \overline{\nabla}^\del_X \phi + \frac12 (\overline{\nabla}_X (\nu \cdot e_0 \cdot)) \, \nu \cdot e_0 \cdot \phi \\
		&= \overline{\nabla}^\del_X \phi + \frac12 (\overline{\nabla}_X \nu) \cdot e_0 \cdot \nu \cdot e_0 \cdot \phi  + \frac12 \nu \cdot (\overline{\nabla}_X e_0) \cdot \nu \cdot e_0 \cdot \phi
	\end{align*}
	for all $X \in \Gamma(T(\del M))$ and $\phi \in \Gamma(\Hyp M_{|\del M})$.
\end{Lem}
\begin{proof}
	Abusing notation, we also denote the corresponding connections on $T\overline{M}_{|\del M}$ by $\overline{\nabla}$ and $\overline{\nabla}^\del$, respectively.
	Their difference defines a tensor $A \coloneqq \overline{\nabla} - \overline{\nabla}^\del \in \Gamma(T^*(\del M) \otimes \mathfrak{so}(T\overline{M},\overline{g})_{|\del M})$.
	Looking at tangential and normal parts separately, it computes to $A_X(Y) = \pi^\mathrm{nor} (\overline{\nabla}_X \pi^\mathrm{tan}(Y)) + \pi^\mathrm{tan}(\overline{\nabla}_X \pi^\mathrm{nor}(Y))$.
	
	Now observe that $T\overline{M}_{|\del M}$ is associated to $P_{\Spin_0(n,1)}\overline{M}_{|\del M}$ via the standard representation $\chi \colon \Spin_0(n,1) \lto \SO_0(n,1) \subseteq \End(\R^{n,1})$.
	If $\tilde{A} \in \Omega^1(P_{\Spin_0(n,1)}\overline{M}_{|\del M}, \mathfrak{spin}(n,1))$ denotes the difference of the connection $1$-forms of the connections $\overline{\nabla}$ and $\overline{\nabla}^\del$ on the spin principal bundle, then $A$ may be expressed as
	\begin{equation*}
		A_X(Y) = [\epsilon,\, \upd \chi(\tilde{A} \circ \upd \epsilon(X)) \,y],
	\end{equation*}
	where $Y = [\epsilon, y] \in T_p \overline{M}$ and $\epsilon$ is a local section of $P_{\Spin_0(n,1)} \overline{M}_{|\del M}$ around $p$ (cf.~\cite[(3.11)]{Baum:2014_2nd}).
	Similarly, the difference term that we aim for is given by
	\begin{equation*}
		(\overline{\nabla}_X - \overline{\nabla}^\del_X) \phi = [\epsilon,\, \upd \rho(\tilde{A} \circ \upd \epsilon(X)) \,\Phi],
	\end{equation*}
	where $\phi = [\epsilon, \Phi]$, $\epsilon$ is as above and $\rho \colon \Cl_{n,1} \to \End(W)$ denotes the Clifford multiplication action.
	
	Around some $p \in \del M$, fix a local orthonormal frame $(e_0, \nu, e_2, \ldots, e_n)$ which admits a lift $\epsilon$.
	Denoting the standard basis of $\R^{n,1}$ by $E_0, E_1, \ldots, E_n$, we obtain $\upd \chi(\tilde{A} \circ \upd \epsilon(X))(E_i) = \sum_j s_j \overline{g}(e_j, A_X(e_i)) E_j$ with $e_1 = \nu$ and $s_j=\overline{g}(e_j,e_j) \in \{\pm 1\}$.
	Using that the isomorphism $\upd \chi$ is given by $E_i E_j \mapsto 2 E_j \langle E_i, - \rangle - 2 E_i \langle E_j, - \rangle$, we obtain
	\begin{equation*}
		\tilde{A} \circ \upd \epsilon(X) = \frac{1}{2} \sum_{i < j} s_i s_j \overline{g}(e_j, A_X(e_i)) E_i E_j.
	\end{equation*}
	Thus, remains to compute
	\begin{align*}
		(\overline{\nabla}_X - \overline{\nabla}^\del_X) \phi &= \frac{1}{2} \sum_{i < j} s_i s_j \overline{g}(e_j, A_X(e_i)) \, e_i \cdot e_j \cdot \phi \\
			&= \frac{1}{2} \sum_{j = 2}^n  - \overline{g}(e_j, \overline{\nabla}_X e_0) \, e_0 \cdot e_j \cdot \phi + \frac{1}{2} \sum_{j = 2}^n \overline{g}(e_j, \overline{\nabla}_X \nu) \, \nu \cdot e_j \cdot \phi \\
			&= -\frac{1}{2} e_0 \cdot (\overline{\nabla}_X e_0) \cdot \phi + \frac{1}{2} \overline{g}(\nu, \overline{\nabla}_X e_0) \, e_0 \cdot \nu \cdot \phi \\
			&\phantom{=}\,+ \frac{1}{2} \nu \cdot (\overline{\nabla}_X \nu) \cdot \phi + \frac{1}{2} \overline{g}(e_0, \overline{\nabla}_X \nu) \, \nu \cdot e_0 \cdot \phi \\
			&= \frac{1}{2} (\overline{\nabla}_X e_0) \cdot e_0 \cdot \phi  - \overline{g}(\nu, \overline{\nabla}_X e_0) \, \nu \cdot e_0 \cdot \phi - \frac{1}{2}  (\overline{\nabla}_X \nu) \cdot \nu \cdot \phi \\
			&= \frac{1}{2} \nu \cdot (\overline{\nabla}_X e_0) \cdot \nu \cdot e_0 \cdot \phi + \frac{1}{2} (\overline{\nabla}_X \nu) \cdot e_0 \cdot \nu \cdot e_0 \cdot \phi. \qedhere
	\end{align*}
\end{proof}

\begin{Bem} \label{Rem:NablaE0}
	The expression $X \mapsto \overline{\nabla}_X e_0$ defines a section of the endomorphism bundle of $TM$ since $\overline{g}(\overline{\nabla}_X e_0, e_0) = \frac{1}{2}\del_X \overline{g}(e_0,e_0) = 0$ for any $X \in TM$.
	Moreover, we have $\overline{g}(\overline{\nabla}_X e_0, Y) = -\overline{g}(e_0, \overline{\nabla}_X Y) = k(X,Y)$ for all $X, Y \in \Gamma(TM)$, so it is the endomorphism associated to $k$ via $g$.
	In particular, it only depends on the initial data set $(g,k)$ and the expression also makes sense when the surrounding Lorentzian manifold $(\overline{M}, \overline{g})$ is not at hand.
\end{Bem}

\section{Dirac-Witten operators as Callias operators} \label{Sec:DWCallias}
In this section, we introduce the main player -- the Dirac-Witten operator -- and study its analytic properties.
As it turns out, the Dirac-Witten operator is a Callias operator, i.\,e.\ of the form Dirac operator plus potential.
The analytic framework will be borrowed from Cecchini and Zeidler \cite{Cecchini.Zeidler:2024}, who studied this kind of operators.

The general setup for this section is the following.
We consider a compact spin manifold $M$ with potentially empty boundary $\del M = \del_+ M \dotcup \del_- M$.
We endow $M$ with an initial data set $(g,k)$ and denote by $\Hyp M$ a hypersurface spinor bundle on $M$ as in \cref{Set:HypSpin}.
As explained in the last section, this carries a connection $\overline{\nabla}_X \phi = \nabla_X \phi - \frac{1}{2} e_0 \cdot k(X,\blank)^\sharp \cdot \phi$ associated to $(g,k)$.
Furthermore, let $\tilde{\nu}$ be the unit normal on $\del M$ that is inward-pointing along $\del_+ M$ and outward-pointing along $\del_- M$.
The function $s$ will be defined to be $+1$ on $\del_+ M$ and $-1$ on $\del_- M$, so that $\nu \coloneqq s \tilde{\nu}$ is inward-pointing on all of $\del M$. 

\begin{Def}
	The \emph{Dirac-Witten operator} $\DW \colon \Gamma(\Hyp M) \to \Gamma(\Hyp M)$ of a hypersurface spinor bundle $\Hyp M \to M$ is defined by the local formula
	\begin{align*}
		\DW &= \sum_{i=1}^n e_i \cdot \overline{\nabla}_{e_i}
	\end{align*}
	where $e_1, \ldots, e_n$ is a local $g$-orthonormal frame.
\end{Def}

A straightforward calculation shows
\begin{align*}
	\DW = D - \frac{1}{2} \tr^g(k)\, e_0 \cdot,
\end{align*}
where the Dirac operator $D = \sum_{i=1}^n e_i \cdot \nabla_{e_i}$ is defined with respect to the connection $\nabla$.
Hence, the Dirac-Witten operator is the sum of a Dirac operator and a potential -- a Callias operator.


One of the most important properties of the Dirac-Witten operator is that it satisfies the following Schrödinger-Lichnerowicz type formula (cf.~\cite{Witten:1981,Parker.Taubes:1982}):
\begin{align} \label{eq:SL}
	\DW^2 = \overline{\nabla}^*\overline{\nabla} + \frac{1}{2} (\rho - e_0 \cdot j^\sharp \cdot),
\end{align}
where \emph{energy density} $\rho$ and \emph{momentum density} $j$ are defined by
\begin{align*}
	\rho &= \frac{1}{2}(\scal^g + \tr^g(k)^2- |k|_g^2) \\
	j &= \div^g(k) - \upd \tr^g(k),
\end{align*}
respectively.
We will study an integrated form of this identity.
For the boundary terms appearing, we use the following definitions:
\begin{Def}
	The \emph{boundary chirality operator} $\bch \colon \Hyp M_{|\del M} \to \Hyp M_{|\del M}$ of the hypersurface spinor bundle is defined by $\bch = \tilde{\nu} \cdot e_0 \cdot = s\nu \cdot e_0 \cdot$.
	The \emph{boundary Dirac-Witten operator} $\bDW \colon  \Gamma(\Hyp M_{|\del M}) \to \Gamma(\Hyp M_{|\del M})$ is defined via the local formula
	\begin{align*}
		\bDW = \sum_{i=2}^n e_i \cdot \nu \cdot \overline{\nabla}^\del_{e_i},
	\end{align*}
	where $\nu, e_2, \ldots, e_n$ is a local $g$-orthonormal frame.
\end{Def}	

It is immediate that $\bch$ is a self-adjoint involution.
We shall need the following properties of $\bDW$.
\begin{Lem}
	The boundary Dirac-Witten operator anti-commutes with the boundary chirality operator.
\end{Lem}
\begin{proof}
	We first observe that $0 = \overline{\nabla} \id_{\Hyp M_{|\del M}} = \overline{\nabla} (\bch^2) = \bch \overline{\nabla} \bch + (\overline{\nabla} \bch) \bch$.
	For any $\phi \in \Gamma(\Hyp M_{|\del M})$ we hence get
	\begin{align*}
		\overline{\nabla}^\del (\bch \phi)
			&= \overline{\nabla} (\bch \phi) - \frac{1}{2} (\overline{\nabla} \bch) \bch\, \bch \phi\\
			& = \overline{\nabla} (\bch) \phi + \bch \overline{\nabla}\phi - \frac{1}{2} (\overline{\nabla} \bch) \bch^2 \phi \\
			&= \bch \overline{\nabla}\phi + \frac{1}{2} (\overline{\nabla} \bch) \bch^2 \phi \\
			&= \bch \overline{\nabla}\phi - \frac{1}{2} \bch (\overline{\nabla} \bch) \bch \phi \\
			&= \bch \overline{\nabla}^\del \phi
	\end{align*}
	using \cref{Lem:BConn}.
	Together with $\nu \cdot \bch = -\bch \nu \cdot$ and $e_i \cdot \bch = \bch e_i \cdot$ for $e_i \perp \nu$, we obtain $\bDW \bch \phi = \sum_{i=2}^n e_i \cdot \nu \cdot \overline{\nabla}^\del_{e_i} (\bch \phi) = \sum_{i=2}^n e_i \cdot \nu \cdot \bch \overline{\nabla}^\del_{e_i} \phi = -\bch \bDW \phi$.
\end{proof}
\begin{Lem} \label{Lem:BDW}
	For $\phi \in \Gamma(\Hyp M)$, we have
	\begin{align*}
		\bDW \phi_{|\del M} &= - \nu \cdot (\DW \phi)_{|\del M} - \overline{\nabla}_\nu \phi + \frac{1}{2} s \left(\tr^{\del M} (-\overline{\nabla}\tilde{\nu}) - \tr^{\del M}(\overline{\nabla}e_0)\,\bch \right) \phi_{|\del M}.
	\end{align*}
\end{Lem}
\begin{proof}
	The necessary calculation is straightforward, using \cref{Lem:BConn} when passing from the first to the second line:
	\begin{align*}
		\bDW \phi_{|\del M} &= \sum_{i=2}^n e_i \cdot \nu \cdot \overline{\nabla}^\del_{e_i} \phi_{|\del M} \\
			&= - \nu \cdot \sum_{i=2}^n e_i \cdot \overline{\nabla}_{e_i} \phi_{|\del M}
				+ \frac{1}{2} \nu \cdot \sum_{i=2}^n e_i \cdot (\overline{\nabla}_{e_i} \bch) \bch \phi_{|\del M} \\
			&= - \nu \cdot (\DW \phi)_{|\del M}
				+\nu \cdot \nu \cdot \overline{\nabla}_\nu \phi \\
			&\phantom{=}\;
				+ \frac{1}{2} \sum_{i=2}^n s \tilde{\nu} \cdot e_i \cdot (\overline{\nabla}_{e_i} \tilde{\nu}) \cdot e_0 \cdot \bch \phi_{|\del M}
				+ \frac{1}{2} \sum_{i=2}^n s \tilde{\nu} \cdot e_i \cdot \tilde{\nu} \cdot (\overline{\nabla}_{e_i}e_0) \cdot \bch \phi_{|\del M} \\
			&= - \nu \cdot (\DW \phi)_{|\del M}
				- \overline{\nabla}_\nu \phi	
				+ \frac{1}{2} s \tr^{\del M}(-\overline{\nabla}\tilde{\nu})\,\tilde{\nu} \cdot e_0 \cdot \bch \phi_{|\del M}
				- \frac{1}{2}s\tr^{\del M}(\overline{\nabla}e_0) \bch \phi_{|\del M} \\
			&= - \nu \cdot (\DW \phi)_{|\del M} - \overline{\nabla}_\nu \phi + \frac{1}{2} s \left(\tr^{\del M} (-\overline{\nabla}\tilde{\nu})  - \tr^{\del M}(\overline{\nabla}e_0)\,\bch \right) \phi_{|\del M} .
		\qedhere
	\end{align*}
\end{proof}

Now, we are ready to state and prove the integrated version of the  Schrö\-ding\-er-Lich\-ne\-ro\-wicz formula for $\DW$.

\begin{Prop} \label{Prop:IntSL}
	For $\phi \in \Gamma(\Hyp M)$, the following holds:
	\begin{align*}
		\|\DW \phi\|_{L^2(M)}^2 &= \|\overline{\nabla} \phi\|_{L^2(M)}^2 + \left( \phi, \frac{1}{2}(\rho - e_0 \cdot j^\sharp \cdot) \phi \right)_{L^2(M)} \\
		&\phantom{=}\,+ \left( \phi_{|\del M}, -\bDW  \phi_{|\del M} + \frac{1}{2}s \left( \tr^{\del M}(-\overline{\nabla}\tilde{\nu}) -  \tr^{\del M}(\overline{\nabla}e_0) \bch \right) \phi_{|\del M} \right)_{L^2(\del M)}.
	\end{align*}
\end{Prop}
\begin{proof}
	The formula follows from taking together four formulae.
	Firstly, there is a partial integration formula for $\DW$.
	This follows from the well-known one for $D$, keeping in mind that the difference term $\DW -D$ is a self-adjoint section in $\Gamma(\End(\Hyp M))$:
	\begin{align*}
		(\DW \phi, \psi)_{L^2(M)} - ( \phi, \DW \psi)_{L^2(M)} &= (D \phi, \psi)_{L^2(M)} - ( \phi, D \psi)_{L^2(M)} \\
		&= (\phi_{|\del M}, \nu \cdot \psi_{|\del M})_{L^2(\del M)}
	\end{align*}
	for all $\phi, \psi \in \Gamma(\Hyp M)$.
	Secondly, the partial integration formula for $\overline{\nabla}$ following from the one for $\nabla$:
	\begin{align*}
		(\overline{\nabla} \phi, \Psi)_{L^2(M)} - ( \phi, \overline{\nabla}^* \Psi)_{L^2(M)} &= (\nabla \phi, \Psi)_{L^2(M)} - ( \phi, \nabla^* \Psi)_{L^2(M)} \\
		&= -(\phi_{|\del M}, \Psi(\nu))_{L^2(\del M)}
	\end{align*}
	for all $\phi \in \Gamma(\Hyp M),\, \Psi \in \Gamma(T^*M \otimes \Hyp M)$.
	Thirdly, there is the Schrödinger-Lichnerowicz formula \eqref{eq:SL}.
	Together, we get:
	\begin{align*}
		\|\DW \phi\|_{L^2(M)}^2 &=  (\phi, \DW^2 \phi)_{L^2(M)} + (\phi_{|\del M}, \nu \cdot (\DW\phi)_{|\del M})_{L^2(\del M)} \\
		&=(\phi,\overline{\nabla}^* \overline{\nabla} \phi)_{L^2(M)} + \left( \phi, \frac{1}{2}(\rho - e_0 \cdot j^\sharp \cdot) \phi \right)_{L^2(M)} \\
		&\phantom{=}\, + (\phi_{|\del M}, \nu \cdot (\DW\phi)_{|\del M})_{L^2(\del M)} \\
		&=\|\overline{\nabla} \phi\|_{L^2(M)}^2 + \left( \phi, \frac{1}{2}(\rho - e_0 \cdot j^\sharp \cdot) \phi \right)_{L^2(M)} \\
		&\phantom{=}\,+ \left( \phi_{|\del M}, \nu \cdot (\DW\phi)_{|\del M} + \overline{\nabla}_\nu \phi \right)_{L^2(\del M)}.
	\end{align*}
	Now the claim follows from the formula of \cref{Lem:BDW}.
\end{proof}

We now consider the Dirac-Witten operator with chirality boundary conditions, i.\,e.\ the Dirac-Witten operator defined on sections $\phi \in \Gamma(\Hyp M)$ with $\bch \phi_{|\del M} = \phi_{|\del M}$.
This fits into the framework of chirality boundary conditions discussed in \cite[Ex.~4.20]{Baer.Ballmann:2016}.
To see this, we first note that $\bDW$ is an adapted boundary operator for $\DW$, i.\,e.\ their respective principal symbols $\sigma_{\bDW}$ and $\sigma_{\DW}$ satisfy $\sigma_{\bDW}(\xi) = \xi^\sharp \cdot \nu \cdot = -\nu \cdot \xi^\sharp \cdot = \sigma_{\DW}(\nu^\flat)^{-1} \circ \sigma_{\DW}(\xi)$ for all $\xi \in T^*\del M$.
Now it just remains to observe that $\bch$ is a self-adjoint involution that anti-commutes with $\bDW$.
In general, chirality boundary conditions are elliptic in the sense of Bär and Ballmann.
Moreover, the chirality boundary condition considered here is also self-adjoint.
This follows from the fact that $\bch$ anti-commutes with $\nu \cdot = \sigma_{\DW}(\nu^\flat)$.

From now on, we assume that the hypersurface spinor bundle $\Hyp M \to M$ is $\Z/2\Z$-graded as in \cref{Set:HypSpin}.
We recall that it carries a scalar product, a metric connection $\nabla$ and a skew-adjoint, $\nabla$-parallel $TM$-Clifford multiplication such that the $\Z/2\Z$-grading is orthogonal, parallel and the Clifford multiplication is odd -- this is what Cecchini and Zeidler call a \emph{$\Z/2\Z$-graded Dirac bundle} over $M$ (cf.~\cite[Def.~2.1]{Cecchini.Zeidler:2024}).
Furthermore, they call a $\Z/2\Z$-graded Dirac bundle over $M$ a \emph{relative Dirac bundle} with \emph{support} $K$ if it is endowed with an odd, self-adjoint and parallel involution $\sigma \in \Gamma(M \setminus K, \End(\Hyp M))$ that anti-commutes with the $TM$-Clifford multiplication and admits a smooth extension to a bundle map on an open neighborhood of $\overline{M \setminus K}$ (\cite[Def.~2.2]{Cecchini.Zeidler:2024}).
Hence, $e_0 \cdot$ equips $\Hyp M$ with the structure of a relative Dirac bundle with empty support.
The formula $\DW = D - \frac{1}{2} \tr^g(k) e_0 \cdot$ shows that the Dirac-Witten operator is the Callias operator (\cite[eq.~(3.1)]{Cecchini.Zeidler:2024}) of this relative Dirac bundle associated to the potential $-\frac{1}{2} \tr^g(k)$, which has also been observed by Chai and Wan \cite{Chai.Wan:2025}.
Here, the Dirac-Witten operator will be viewed as bounded operator
\begin{align*}
	\DW_\bch \colon H^1_\bch(\Hyp M) \to L^2(\Hyp M),
\end{align*}
where $H^1_\bch(\Hyp M)$ is the closure of $\{\phi \in \Gamma(\Hyp M) \mid \bch \phi_{|\del M} = \phi_{|\del M}\}$ with respect to the $H^1$-Sobolev norm.
The analytic results from \cite[Sec.~3]{Cecchini.Zeidler:2024} give the following proposition.

\begin{Prop}[{\cite[Thm.~3.4]{Cecchini.Zeidler:2024}}]
	The Dirac-Witten operator with chirality boundary conditions defines a Fredholm operator $\DW_\bch \colon H^1_\bch(\Hyp M) \to L^2(\Hyp M)$, which is self-adjoint when considered as unbounded operator $L^2(\Hyp M) \supseteq H^1_\bch(\Hyp M) \to L^2(\Hyp M)$.
\end{Prop}

Since $M$ is compact and by homotopy invariance of the index, we expect its $\Z/2\Z$-graded index
\begin{align*}
	\mathrm{ind}_{\Z/2\Z}(\DW_\bch) \coloneqq \tr(\iota_{|\ker(\DW_\bch)}) = \dim \ker({\DW_\bch}_{|\Hyp^+ M}) - \dim \ker({\DW_\bch}_{|\Hyp^- M})
\end{align*}
to be independent of $(g,k)$.
Here, $\iota$ denotes the $\Z/2\Z$-grading operator and $\Hyp^\pm M$ its $\pm 1$-eigenspaces. 
In fact, it can be expressed by a topological formula that we discuss in the case of the “classical” hypersurface spinor bundle.
If $n$ is even, then $i^{\frac{n+2}{2}}\iota e_0 \cdot e_1 \cdot \dots e_n \cdot$ defines an additional symmetry of any $\Z/2\Z$-graded hypersurface spinor bundle $\Hyp M$, forcing the index to be zero.
We can thus restrict our attention to the odd-dimensional case.
In this case there is a unique irreducible representation of $\Cl_{n,1}$.
It can be endowed with the $\Z/2\Z$-grading induced by the volume form.
The \emph{classical hypersurface spinor bundle} is the $\Z/2\Z$-graded hypersurface spinor bundle associated to this representation.

\begin{Bem}
	Let us denote by $\Hyp M = \Hyp^+ M \oplus \Hyp^- M$ the decomposition of the classical hypersurface spinor bundle given by the $\Z/2\Z$-grading for $n$ odd.
	Then we can identify $\Hyp^- M \overset{\cong}{\to} \Hyp^+ M$ via $ie_0 \cdot$.
	The involution $e_0 \cdot$ then corresponds to the matrix
	\begin{align*}
		\begin{pmatrix}
			0 & -i \\ i & 0
		\end{pmatrix}.
	\end{align*}
	The Clifford multiplication by $X$ gets identified with
	\begin{align*}
		\begin{pmatrix}
			0 & ie_0 \cdot X \cdot \\ i e_0 \cdot X \cdot & 0
		\end{pmatrix}.
	\end{align*} 
	The operators $ie_0 \cdot X \cdot$ define a Clifford multiplication on $\Hyp^+ M$.
	In fact, $\Hyp^+ M \to M$ is associated to a $\Cl_n$-representation obtained by suitably restricting the irreducible $\Cl_{n,1}$-representation.
	For dimension reasons, this $\Cl_n$-representation is irreducible, so $\Hyp^+ M$ is associated to one of the two irreducible representations of $\Cl_n$.
	In the description of the classical hypersurface spinor bundle, we had been vague about what we mean by the volume element, which determines the decomposition $\Hyp M = \Hyp^+ M \oplus \Hyp^- M$.
	We choose it in such a way that $\Hyp^+ M \to M$ becomes isomorphic to the classical spinor bundle on $M$ considered in \cite[Ex.~2.6]{Cecchini.Zeidler:2024}, which appears to be characterized by $i^{\frac{n+1}{2}}e_1 \cdot \ldots e_n \cdot = \id$ for any positively oriented orthonormal frame $e_1, \ldots, e_n$.
	So the volume element we take is $i^{\frac{n+1}{2}} i e_0 \cdot e_1 \cdot \ldots i e_0 \cdot e_n =i^{\frac{n+1}{2} + 1}e_0 \cdot e_1 \cdot \ldots e_n$.
	We observe that then $\Hyp M \to M$ recovers the (untwisted) relative Dirac bundle of the cited example.
	Note that while the construction by Cecchini and Zeidler appears rather ad-hoc, the hypersurface spinor bundle has a geometric meaning.
	In particular, the involution (called $\sigma$ there) now naturally arises as Clifford multiplication with the unit normal.
\end{Bem}

Now we are ready to state the index theorem for the Dirac-Witten operator.
Note that the formula is specific for the classical hypersurface spinor bundle.
As a consequence, we obtain a criterion for the existence of non-trivial elements in the kernel of the Dirac-Witten operator, which also holds independently of the chosen hypersurface spinor bundle (and even when it is not explicitly $\Z/2\Z$-graded).

\begin{Satz}[Callias Index Theorem, {\cite[Cor.~3.10]{Cecchini.Zeidler:2024}}]
	The index of $\DW_\bch \colon H^1_\bch(\Hyp M) \to L^2(\Hyp M)$ is given by $\mathrm{ind}_{\Z/2\Z}(\DW_\bch) = \hat{A}(\del_- M)$.
\end{Satz}

\begin{Kor} \label{Cor:Callias}
	Assume that $\hat{A}(\del_- M) \neq 0$.
	Then there are non-trivial smooth \emph{Dirac-Witten harmonic spinors} subject to the chirality boundary condition, i.\,e.\ $\phi \in \Gamma(\Hyp M) \setminus \{0\}$ with $\DW \phi = 0$ and $\bch \phi_{|\del M} = \phi_{|\del M}$.
\end{Kor}

\section{The kernel of Dirac-Witten operators} \label{Sec:KerDW}
In this section, we investigate some consequences of non-zero kernel of the Dirac-Witten operator on a compact spin manifold with boundary.
To a large extent this discussion is similar to the closed case that was treated in \cite{Ammann.Gloeckle:2023}.

The general setup for this section is the following.
We consider an initial data set $(g,k)$ on a compact spin manifold $M$ with boundary $\del M = \del_+ M \dotcup \del_- M$.
We denote by $\tilde{\nu}$ the unit normal on $\del M$ that is inward-pointing along $\del_+ M$ and outward-pointing along $\del_- M$.
The function $s$ will be defined to be $+1$ on $\del_+ M$ and $-1$ on $\del_- M$.
Furthermore, we consider a hypersurface spinor bundle $\Hyp M$ on $M$ with its connection $\overline{\nabla}_X \phi = \nabla_X \phi - \frac{1}{2} e_0 \cdot k(X,\blank)^\sharp \cdot \phi$ associated to $(g,k)$ (cf.~\cref{Set:HypSpin}).
Its Dirac-Witten operator with respect to chirality boundary conditions $\tilde{\nu} \cdot e_0 \cdot \phi_{|\del M} = \phi_{|\del M}$ will be denoted by $\DW_\bch$.

\begin{Prop} \label{Prop:Kernel}
Let $M$ be as above and $(g,k)$ an initial data set on $M$.
We assume that it is subject to the dominant energy condition $\rho \geq |j|_g$ and that the future outgoing null expansion scalar $\theta^+ = \tr^{\del M}(\nabla \tilde{\nu}) + \tr^{\del M}(k)$ (with respect to $\tilde\nu$) satisfies $\theta^+ \leq 0$ on $\del_+M$ and $\theta^+ \geq 0$ on $\del_- M$.
Then any $\phi \in \ker(\DW_\bch)$ is $\overline{\nabla}$-parallel and satisfies $(\rho e_0 - j^\sharp) \cdot \phi = 0$.
If, moreover, $M$ is connected and $\phi \not\equiv 0$, then $\phi$ is nowhere vanishing, $\rho = |j|_g$ and $\theta^+ = 0$.
\end{Prop}
\begin{proof}
We consider the integrated Schrödinger-Lichnerowicz type formula from \cref{Prop:IntSL}
\begin{align*}
		\|\DW \phi\|_{L^2(M)}^2 &= \|\overline{\nabla} \phi\|_{L^2(M)}^2 + \frac{1}{2}\left( \phi, (\rho - e_0 \cdot j^\sharp \cdot) \phi \right)_{L^2(M)} + \left( \phi_{|\del M}, -\bDW  \phi_{|\del M} \right)_{L^2(\del M)} \\
&\phantom{=}\,+ \frac{1}{2} \left(\phi_{|\del M}, s \left( \tr^{\del M}(-\overline{\nabla}\tilde{\nu}) -  \tr^{\del M}(\overline{\nabla}e_0) \bch \right) \phi_{|\del M} \right)_{L^2(\del M)}.
\end{align*}
As $\phi$ is subject to the boundary condition $\bch \phi_{|\del M} = \phi_{|\del M}$ and $\bDW$ anti-commutes with $\bch$, we have
\begin{align*}
	\left( \phi_{|\del M}, -\bDW  \phi_{|\del M} \right)_{L^2(\del M)} &= \left( \phi_{|\del M}, -\bDW  \bch \phi_{|\del M} \right)_{L^2(\del M)} = \left( \phi_{|\del M}, \bch  \bDW \phi_{|\del M} \right)_{L^2(\del M)}\\
	&= \left(\bch \phi_{|\del M}, \bDW  \phi_{|\del M} \right)_{L^2(\del M)} = \left(\phi_{|\del M}, \bDW  \phi_{|\del M} \right)_{L^2(\del M)}
\end{align*}
and hence the third summand on the right is zero.
Recalling in addition that $\theta^+ = \tr^{\del M}(\nabla \tilde{\nu}) + \tr^{\del M}(\overline{\nabla} e_0)$ (cf.~\cref{Rem:NablaE0}), the formula simplifies to
\begin{align*}
\|\DW \phi\|^2_{L^2(M)} &= \|\overline{\nabla} \phi\|^2_{L^2(M)} + \frac{1}{2}  \left( \phi, (\rho - e_0 \cdot j^\sharp \cdot) \phi \right)_{L^2(M)}  + \frac{1}{2} \left( \phi_{|\del M}, -s \theta^+ \phi_{|\del M} \right)_{L^2(\del M)}.
\end{align*}
Clearly, the first term is always non-negative, the second one is non-negative if the dominant energy condition holds and the third term is non-negative by our assumptions as well.
Hence, all these terms must be zero if $\DW \phi = 0$, in particular $\overline{\nabla} \phi = 0$.
It then follows from the Schrödinger-Lichnerowicz formula that $\frac{1}{2}(\rho - e_0 \cdot j^\sharp \cdot) \phi = \DW^2 \phi - \overline{\nabla}^* \overline{\nabla} \phi = 0$.

If $M$ is connected and $\phi \not\equiv 0$, then $\overline{\nabla}$-parallelism of $\phi$ implies that $\phi$ is nowhere vanishing.
Then we get $\rho = |j|_g$ and $\theta^+ = 0$ since the latter two terms in the equation above are zero and $\rho \geq |j|_g$ and $-s\theta^+ \geq 0$, respectively.
\end{proof}

Even more can be deduced by looking at the Dirac current of $\phi$.
To define it, we use the Lorentzian metric $\overline{g}$ on $TM \oplus \R e_0$ defined by $\overline{g}(U+u e_0, U^\prime + u^\prime e_0) = g(U, U^\prime) -u u^\prime$ for all $U, U^\prime \in T_pM$, $p \in M$ and $u, u^\prime \in \R$.
\begin{Def} \label{Def:DirCurr}
The \emph{(Lorentzian) Dirac current} associated to $\phi \in \Hyp_pM$, $p \in M$, is the vector $V_\phi \in T_pM \oplus \R e_0$ uniquely determined by the condition
\begin{align*}
\overline{g}(V_\phi, X) = - \langle e_0 \cdot X \cdot \phi, \phi \rangle
\end{align*}
for all $X \in T_pM \oplus \R e_0$.
Its \emph{Riemannian Dirac current} $U_\phi \in T_pM$ is defined by
\begin{align*}
g(U_\phi, X) = \langle e_0 \cdot X \cdot \phi, \phi \rangle
\end{align*}
for all $X \in T_pM$.
\end{Def}
Since $\overline{g}(V_\phi,X) = \overline{g}(-U_\phi,X)$ for all $X \in TM$ and $\overline{g}(V_\phi,e_0) = -\langle e_0 \cdot e_0 \cdot \phi, \phi \rangle = -|\phi|^2$, the Lorentzian Dirac current splits up as $V_\phi = u_\phi e_0 - U_\phi$ with $u_\phi = |\phi|^2$.
Thus $V_\phi$ is zero if and only if $\phi = 0$.
Moreover, a short calculation shows $|V_\phi \cdot \phi|^2 = -\overline{g}(V_\phi,V_\phi)|\phi|^2$ (cf.\ e.\,g.~\cite[Lem.~1.3.8]{Gloeckle:2024_b}).
Hence if $\phi \neq 0$, then $V_\phi$ is either future-timelike or future-lightlike.
In the latter case, additionally $V_\phi \cdot \phi = 0$ holds.

If $\phi \in \Gamma(\Hyp M)$ is $\overline{\nabla}$-parallel, then $\overline{\nabla} V_\phi = 0$ or, equivalently,
\begin{equation} \label{eq:VPar}
	\begin{aligned}
		\nabla_X U_\phi &= u_\phi k(X,\blank)^\sharp \\
		\upd u_\phi(X) &= k(U_\phi,X)
	\end{aligned}
\end{equation}
for all $X \in \Gamma(TM)$.
In particular, whether $V_\phi$ is zero, timelike or lightlike will not change on a connected component of $M$.
\begin{Def}
A $\overline{\nabla}$-parallel spinor $\phi \in \Gamma(\Hyp M)$ is called \emph{timelike} or \emph{lightlike} if $V_\phi$ is timelike or lightlike on all of $M$, respectively.
\end{Def}

\begin{Prop} \label{Prop:KernelCurrent}
Let $(g,k)$ be an initial data set on a compact connected spin manifold $M$ with non-empty boundary $\del M = \del_+ M \dotcup \del_-M$.
As above, we consider the unit normal $\tilde{\nu}$ on $\del M$ that is inward-pointing along $\del_+ M$ and outward-pointing along $\del_- M$.
Assume that $\phi$ is a non-zero $\overline{\nabla}$-parallel spinor on $M$ subject to chirality boundary conditions.
Then $\phi$ is a lightlike $\overline{\nabla}$-parallel spinor.
Its Riemannian Dirac current $U_\phi$ satisfies  $j^\sharp = \frac{\rho}{u_\phi} U_\phi$ on all of $M$ and $\tilde{\nu} = -\frac{1}{u_\phi} U_\phi$ on the boundary $\del M$.
Moreover, $U_\phi$ is a non-vanishing vector field with $\upd U_\phi^\flat = 0$.
In particular, $\ker(U_\phi^\flat)$ defines an involutive distribution and hence a foliation of $M$ by Frobenius' theorem.
The leaves may be co-oriented by the unit normal $\tilde{\nu} = -\frac{1}{u_\phi}U_\phi$, and then the future outgoing null second fundamental form $\chi^+ = \nabla \tilde{\nu}^\flat + k$ satisfies $\chi^+ = 0$, in particular the leaves are MOTS.
On the leaves, the restriction of $\frac{\phi}{|\phi|}$ is parallel with respect to the Levi-Civita connection of the induced metric.
In particular, the induced metric on every leaf is Ricci-flat.
\end{Prop}

We use the following lemma.
\begin{Lem}
Assume $V \in T_pM \oplus \R e_0$ and $\phi \in \Hyp_p M \setminus \{0\}$ is some spinor with $V \cdot \phi = 0$.
Then $V$ is a scalar multiple of $V_\phi$.
If additionally $V \neq 0$, then $V_\phi$ is lightlike.
\end{Lem}
\begin{proof}
We have $\overline{g}(V_\phi,V_\phi) \leq 0$, $ \overline{g}(V,V) \phi = -V \cdot V \cdot \phi = 0$ and $\overline{g}(V_\phi,V) = -\langle e_0 \cdot V \cdot \phi, \phi \rangle = 0$.
Hence, $\overline{g}_{|L \times L}$ is negative semi-definite for $L = \spn(V, V_\phi) \subseteq T_pM \oplus \R e_0$.
Since $\overline{g}$ is a Lorentzian metric, the dimension of $L$ can be at most one, yielding the first part.
If now $V \neq 0$, then $L$ is a one-dimensional lightlike subspace and the rest of the claim follows.
\end{proof}

\begin{proof}[Proof of \Cref{Prop:KernelCurrent}]
On $\del M$ the boundary condition yields $(e_0 + \tilde{\nu}) \cdot \phi = 0$.
Using the previous lemma and $\del M \neq \emptyset$ we obtain that the $\overline{\nabla}$-parallel spinor $\phi$ is lightlike with $V_\phi = u_\phi (e_0 + \tilde{\nu})$ along the boundary.
A further application of the lemma, this time to $(\rho e_0 - j^\sharp) \cdot \phi = 2e_0 \cdot (\DW^2 \phi - \overline{\nabla}^*\overline{\nabla} \phi) = 0$, yields $\rho e_0 - j^\sharp = \frac{\rho}{u_\phi} V_\phi$.

As $V_\phi$ is lightlike, $|U_\phi|_g = u_\phi = |\phi|^2$, so $U_\phi$ is nowhere vanishing.
From $\overline{\nabla}$-parallelism of $V_\phi$ (cf.~\eqref{eq:VPar}), we get
\begin{align*}
\upd U_\phi^\flat(X,Y) &=  \del_X (U_\phi^\flat(Y)) - \del_Y (U_\phi^\flat(X)) - U_\phi^\flat([X,Y]) \\
	&=g(\nabla_X U_\phi, Y) - g(\nabla_Y U_\phi, X) \\
	&= u_\phi k(X,Y) - u_\phi k(Y,X) = 0
\end{align*}
for all $X, Y \in \Gamma(TM)$.
The expression in the first line of this equation then helps to conclude that $\ker(U_\phi^\flat)$ is an involutive distribution of codimension one.

Let us now calculate $\chi^+$ of the leaves of the associated foliation, where the unit normal on the leaves is given by $\tilde{\nu} \coloneqq -\frac{1}{u_\phi} U_\phi$.
(On the boundary, this coincides with the previously defined $\tilde{\nu}$.)
We have
\begin{align*}
\nabla_X \tilde{\nu} &= \bafrac{1}{u_\phi^2} \upd u_\phi(X) U_\phi - \frac{1}{u_\phi} \nabla_X U_\phi \\
	&= \bafrac{1}{u_\phi^2} k(U_\phi,X) U_\phi - k(X,\blank)^\sharp \\
	&= -(k(X, \blank)^\sharp - k(X,\tilde{\nu}) \tilde{\nu})
\end{align*}
Thus $\chi^+(X,Y) = g(\nabla_X \tilde{\nu}, Y) + k(X,Y) = 0$ for all $X,Y$ tangential to the leaves of the foliation.

Let now $F$ be a leaf and $X \in TF$.
The Levi-Civita connection of $F$ induces on $\Hyp M_{|F}$ the connection $\nabla^F_X = \nabla_X + \frac{1}{2} \nabla_X \tilde{\nu} \cdot \tilde{\nu} \cdot$, cf.~\eqref{eq:SpBConn}.
Thus, using $\overline{\nabla}_X \phi = 0$ and $(e_0+ \tilde{\nu}) \cdot \phi = \frac{1}{u_\phi} V_\phi \cdot \phi = 0$, we obtain
\begin{align*}
\nabla^F_X \phi &= \nabla_X \phi + \frac{1}{2} \nabla_X \tilde{\nu} \cdot \tilde{\nu} \cdot \phi \\
	&= \overline{\nabla}_X \phi - \frac{1}{2} k(X,\blank)^\sharp \cdot e_0 \cdot \phi 
		- \frac{1}{2} \left(k(X, \blank)^\sharp - k(X,\tilde{\nu}) \tilde{\nu} \right) \cdot \tilde{\nu} \cdot \phi \\
	&= -\frac{1}{2} k(X,\tilde{\nu}) \phi.
\end{align*}
Hence,
\begin{align*}
\nabla^F_X \frac{\phi}{|\phi|} &= \nabla^F_X \left(u_\phi^{-\frac{1}{2}} \phi \right) \\
	&= -\frac{1}{2} u_\phi^{-\frac{3}{2}} \upd u_\phi(X) \phi - \frac{1}{2}u_\phi^{-\frac{1}{2}} k(X, \tilde{\nu}) \phi \\
	&= 0
\end{align*}
as desired.
\end{proof}

In the next section, we will prove the following general fact about foliations.
Although its proof only uses elementary differentialgeometric methods, its setup of foliations on manifolds with boundary is rather special and not covered by standard textbooks.
 
\begin{Satz} \label{Thm:ProdStr}
	Let $M$ be a connected manifold with boundary $\del M = \del_+ M \dotcup \del_- M$, where $\del_+ M$ and $\del_- M$ are unions of components and $\del_+ M \neq \emptyset$.
	Let $U$ be a non-vanishing vector field on $M$ that is outward-pointing on $\del_+ M$ and inward-pointing on $\del_- M$.
	We assume that there exists a metric $g$ on $M$ such that $U$ is orthogonal to the boundary and its metric dual satisfies $\upd U^\flat = 0$.
	Then the flow of $X = -\bafrac{U}{|U|_g^2}$ defines a diffeomorphism
	\begin{align*}
	\Phi^\prime \colon \del_+ M \times [0,\ell] \to M
	\end{align*}
	for some $\ell > 0$.
	Moreover, $\Phi^\prime$ maps the leaves of the foliation $(\del_+ M \times \{t\})_{t \in \R}$ precisely to the leaves of the foliation defined by $U^\flat$.
\end{Satz}

With this result at hand, we can prove the main theorems of this article.
\begin{proof}[Proof of \Cref{Thm:Main} and \cref{Add:Main}]
	Since $\hat{A}(\del_- M) \neq 0$, the dimension of $\del_- M$ is even and the one of $M$ is odd.
	Let $\Hyp M \to M$ be the irreducible hypersurface spinor bundle on $M$.
	It follows from the Callias index theorem (cf.~\cref{Cor:Callias}) that there is a spinor $\phi \in \Gamma(\Hyp M) \setminus \{0\}$ with $\DW \phi = 0$ and $\bch \phi_{|\del M} = \phi_{|\del M}$ for the initial data set $(g,k)$.
	From \Cref{Prop:Kernel}, we get that $\phi$ is a non-zero $\overline{\nabla}$-parallel spinor.
	
	Now note that $\del_+ M \neq \emptyset$, as otherwise $\del_- M$ would be spin zero-bordant and $\hat{A}(\del_- M) = 0$.
	We apply \cref{Prop:KernelCurrent} to $\phi$ and obtain that $\phi$ is a lightlike $\overline{\nabla}$-parallel spinor and its Riemannian Dirac current $U_\phi$ is nowhere vanishing, outward-pointing along $\del_+ M$, inward-pointing along $\del_- M$ and satisfies $\upd U_\phi^\flat = 0$.
	Now \cref{Thm:ProdStr} provides us with a diffeomorphism $\Phi^\prime \colon \del_+ M \times [0,\ell] \to M$.
	Using the identification $\del_+ M \cong \Phi^\prime(\del_+ M \times \{\ell\}) = \del_- M$, $p \mapsto \Phi^\prime(p,\ell)$, we obtain a diffeomorphism $\Phi \colon \del_- M \times [0,\ell] \to M$ and we claim that this has all the desired properties.
	
	By construction, $(\Phi(\del_- M \times \{t\}))_{t \in [0, \ell]}$ coincides with the foliation defined by $U_\phi^\flat$.
	This directly shows the claim of the addendum and allows us to make use of the properties derived in \cref{Prop:KernelCurrent}:
	Those are that each leaf satisfies $\chi^+ = 0$ with respect to the unit normal $\tilde{\nu} = -\frac{1}{u_\phi} U_\phi$, carries the parallel spinor $\frac{\phi}{|\phi|}$ and is orthogonal to $j^\sharp = -\rho \tilde{\nu}$.	
	This is all that was left to show for the theorem.
\end{proof}

\begin{proof}[Proof of \cref{Thm:Riem}]
	We consider the initial data set $(g, -(h \circ s)\cdot g)$ on $M$.
	For this initial data set, we have
	\begin{align*}
	 	2\rho &= \scal^g + n(n-1) (h \circ s)^2 \quad\text{and} \\
	 	j &= (n-1) (h^\prime \circ s) \upd s.
	 \end{align*}
	 Since
	 \begin{gather*}
	 	|j|_g = (n-1) |h^\prime \circ s| |\upd s|_g \leq -(n-1)(h^\prime \circ s),
	 \end{gather*}
	the inequality for $\scal^g$ implies the dominant energy condition $\rho \geq |j|_g$.
	Moreover, keeping in mind that $H^g$ is defined with respect to the inward-pointing unit normal $\nu$, we have
	\begin{align*}
		\theta^+ &= \tr^{\del_+ M} (\nabla \tilde{\nu} + k) = (n-1) (-H^g - h(0)) \leq 0 &\text{on}\;&\del_+ M \;\;\text{and} \\
		\theta^+ &= \tr^{\del_- M} (\nabla \tilde{\nu} + k) = (n-1) \!\!\phantom{-}(H^g - h(L))\,\, \geq 0 &\text{on}\;&\del_- M.
	\end{align*}
	Hence, the assumptions of \cref{Thm:Main} are satisfied for this initial data set.
	
	We get a diffeomorphism $\tilde{\Phi} \colon \del_- M \times [0, \tilde{\ell}]$ inducing the foliation defined by $U_\phi^\flat$ for a lightlike $\overline{\nabla}$-parallel spinor $\phi$.
	If $X$ is a vector tangential to the leaves, then $\upd u_\phi(X) = k(U_\phi, X) = -(h \circ s)g(U_\phi, X) = 0$, so $u_\phi$ is constant along the leaves.
	Hence, we can reparameterize the second factor to obtain a diffeomorphism $\Phi \colon \del_- M \times [0,\ell] \to M$ such that $\Phi_*(\frac{\del}{\del t}) = \tilde{\nu}$.
	Then $\Phi^* g = \gamma_t + \upd t^2$ for a family $(\gamma_t)_{t \in [0,\ell]}$ of metrics on $\del_- M$.
	
	\Cref{Thm:Main} provides us with more knowledge about the leaves $F_t = \Phi(\del_- M \times \{t\})$.
	First, $j^\sharp = - \rho \tilde{\nu}$ shows that $(n-1)\del_X(h \circ s) = j(X) = 0$ for all $X \in TF_t$.
	Thus $h \circ s$ is constant along the leaves and we may define $(h_t)_{t \in [0, \ell]}$ by $\{h_t\} = (h \circ s)(F_t)$ for all $t \in [0, \ell]$.
	Moreover, since $\rho = |j|_g$ the inequalities used to establish the dominant energy condition must be equalities.
	This yields the scalar curvature equality as well as $|\upd s|_g = 1$ whenever $h^\prime \circ s \neq 0$.
	Since $-j^\sharp$ and $\tilde{\nu}$ point in the same direction, $\upd t = \Phi^*(\upd s)$ holds where $h^\prime \circ s \neq 0$.
	
	Second, $0 = \chi^+(X,Y) = \gamma_t(\nabla_X \tilde{\nu}, Y) - h_t \cdot \gamma_t(X,Y)$ for all $X,Y \in TF_t$ implies
	\begin{gather*}
		\frac{\del}{\del t}\gamma_t = 2\gamma_t(\nabla \tilde{\nu}, -) = 2h_t \cdot \gamma_t.
	\end{gather*}
	The unique solution of this ordinary differential equation starting at $\gamma_0 = \gamma$ is given by $\gamma_t = v(t)^2 \gamma$ with $v(t) = \exp(\int_0^t h_\tau \upd \tau)$.
	Moreover, since $\chi^+ = 0$ implies $\theta^+ = 0$, the inequalities for $H^g$ along $\del M$ turn into equalities.
	
	Third, the metrics on the leaves admit a non-trivial parallel spinor.
	In particular, this applies to $\gamma$, which was left to show.
\end{proof}

\section{Identifying the product structure} \label{Sec:ProdStr}
This section is devoted to a proof of \cref{Thm:ProdStr}.
Throughout, $M$ will be a compact manifold with boundary $\del M = \del_+ M \dotcup \del_- M$, where $\del_+ M$ and $\del_-M$ are unions of components.
Moreover, we will assume that $M$ is connected and that $\del_+ M \neq \emptyset$.
For a manifold with boundary, as usual, notions such as diffeomorphism, foliation etc.\ should always be understood in the sense that there exists a smooth extension along a small collar neighborhood around the boundary.
Notice that with this notion, the leaves of the foliation of $M$ are not necessarily homeomorphic:
Even for two diffeomorphic leaves of the collar-extended manifold, their portions lying inside $M$ may be topologically different.

\begin{Satz} \label{Thm:FlowFol}
Let $M$ be as above and $X$ a smooth vector field that is transverse to the boundary, inward-pointing on $\del_+ M$ and outward-pointing on $\del_- M$.
We denote by
\begin{align*}
	\Phi \colon \mathcal{D} \to M,
\end{align*}
the flow of $X$, defined on the maximal domain of definition $\mathcal{D} \subseteq M \times \R$.
Then $\Phi$ restricts to a diffeomorphism
\begin{align*}
	\Phi^\prime \colon \mathcal{D}^\prime \coloneqq (\del_+ M \times [0,\infty)) \cap \mathcal{D} \to M^\prime.
\end{align*}
onto an open subset $M^\prime \subseteq M$, and the smooth function $f \coloneqq \mathrm{pr}_\R \circ {\Phi^\prime}^{-1} \colon M^\prime \to \R$ is proper.
Moreover, if the foliation $(f^{-1}(t))_{t \in \R}$ of $M^\prime$ extends to a foliation by hypersurfaces of $M$ such that $X$ is transversal to its leaves, then $M^\prime = M$.
\end{Satz}

Note that the additional condition of $X$ being transversal to some foliation by hypersurfaces forces the vector field $X$ to be nowhere vanishing, which already rules out many pathological examples.
Yet $X$ being nowhere vanishing is not sufficient for the full conclusion:
Even in nice cases $M^\prime$ might not be all of $M$ as the following example shows.

\begin{Bsp}
	Consider $M = S^1 \times [-\ell,\ell]$ with $\del_+ M = S^1 \times \{\ell\}$ and $\del_- M = S^1 \times \{-\ell\}$ and the vector field $X(s,r) = \frac{\del}{\del s} - r^2 \frac{\del}{\del r}$, $s \in S^1, r \in [-\ell, \ell]$.
	In this case $M^\prime = S^1 \times (0, \ell]$ and $f(s,r) = \frac{1}{r}-\frac{1}{\ell}$.
	Since $f$ is independent on $s$, the foliation $(f^{-1}(t))_{t \in \R}$ of $M^\prime$ extends to the canonical foliation of $M$.
	But note that $X$ is not transversal to the leaf $S^1 \times \{0\}$.
\end{Bsp}

\begin{proof}
Let $\hat{M}$ be the manifold (without boundary) that arises by adding collar neighborhoods to $M$, $\hat{X}$ a smooth extension of $X$ to $\hat{M}$ and $\hat{\Phi} \colon \hat{\mathcal{D}} \to \hat{M}$ be the flow of $\hat{X}$ defined on the maximal domain of definition.
Note that $\Phi$ is the restriction of $\hat{\Phi}$ to $\hat{\Phi}^{-1}(M) \cap (M \times \R)$.
This uses that the vector field $X$ is transversal to the boundary, so that flow lines of $\hat{X}$ cannot re-enter $M$.

We start by showing that $\Phi^\prime$, or rather $\hat{\Phi}^\prime \coloneqq \hat{\Phi}_{|\hat{\mathcal{D}}^\prime}$ for $\hat{\mathcal{D}}^\prime = (\del_+ M \times \R) \cap \hat{\mathcal{D}}$, is a local diffeomorphism.
By definition, for $x \in \del_+ M$, $Y \in T_x\del_+ M$ and $a \in \R$ the differential of $\hat{\Phi}^\prime$ is given by $\upd_x \hat{\Phi}^\prime(Y+a\frac{\del}{\del t}) = Y + aX$.
As $X$ is transversal to $\del_+ M$, this differential is an isomorphism and $\hat{\Phi}^\prime$ is a local diffeomorphism on a neighborhood of $\del_+ M \times \{0\}$ in $\hat{\mathcal{D}}^\prime$.
We now consider arbitrary points $(x,t) \in \hat{\mathcal{D}}^\prime$.
Using that, locally around $(x,0)$, $\hat{\Phi}(\blank,t)$ is a diffeomorphism (with inverse $\hat{\Phi}(\blank,-t)$) and the factorization $\hat{\Phi}^\prime=\hat{\Phi}(\blank,t) \circ \hat{\Phi}^\prime \circ ((y,s) \mapsto (y,s-t))$, we conclude that $\hat{\Phi}^\prime$ is also a local diffeomorphism around $(x,t)$.

For the first part of the claim it is now sufficient to see that $\hat{\Phi}^\prime$ is injective.
Then $\hat{\Phi}^\prime$ is a diffeomorphism onto its image and $\Phi^\prime$ its restriction to $\mathcal{D}^\prime$.
Suppose that $\hat{\Phi}^\prime(x,t) = \hat{\Phi}^\prime(y,s)$ and $t \leq s$.
Then $x=\hat{\Phi}(\blank,-t)\circ \hat{\Phi}^\prime(x,t) = \hat{\Phi}^\prime(y,s-t)$.
As $X$ is inward-pointing at $x \in \del_+ M$, this implies $s-t=0$ and $x=\hat{\Phi}^\prime(y,0)=y$.

We now show that the subsets $f^{-1}([a,b])$ are compact for any real numbers $a \leq b$.
Let $(x_i)_{i \in \N}$ be a sequence in $f^{-1}([a,b]) \subseteq M^\prime$.
As $\del_+ M \times [a,b]$ is compact, we may assume without loss of generality that $(y_i, c_i) \coloneqq (\Phi^\prime)^{-1}(x_i) \lto (y,c)$ for some $(y,c) \in \del_+ M \times [a,b]$.
We have to show that $(y,c) \in \mathcal{D}^\prime$.
Then $x = \Phi(y,c) \in M^\prime$ is the limit of the sequence $(x_i)_{i \in \N}$.

Suppose for contradiction that $(y,c) \not\in \mathcal{D}$.
Let $d$ be the maximal value so that $(y,d) \in \mathcal{D}$.
Since $M$ is compact, this maximum exists.
Furthermore, we may choose an $\epsilon > 0$ smaller than $c-d$ such that $(y, d+\epsilon) \in \hat{\mathcal{D}}$ with $\hat{\Phi}(y, d+\epsilon) \in \hat{M} \setminus M$.
Thus $d+\epsilon \leq c_i$ for almost all $i \in \N$, so that $(y_i, d + \epsilon) \in \mathcal{D}$ and $\Phi(y_i, d + \epsilon) \in M$.
But since $M$ is compact, thus closed in $\hat{M}$, we deduce from $\hat{\Phi}(y_i, d + \epsilon) \lto \hat{\Phi}(y, d + \epsilon)$ for $i \lto \infty$ that $\hat{\Phi}(y, d + \epsilon) \in M$, contradiction.

For the last part, we show that $M^\prime$ is closed in $M$ and invoke that $M$ is connected.
So let $x \in \overline{M^\prime}$.
We may assume without loss of generality that $x$ is in the interior of $M$:
If $x \in \del_+ M$, there is nothing to show as $x \in M^\prime$ by definition, and if $x \in \del_- M$, we may argue with $\Phi(x,-\epsilon) \in M \setminus \del M$ for a small $\epsilon > 0$ instead as this will be in $\overline{M^\prime}$ if $x$ is.
By assumption, there is a codimension one foliation $\mathcal{F}$ of $M$ that is transversal to $X$ and such that its leaves are level sets of $f$ wherever $f$ is defined.
We may choose a chart $\psi \colon U \to (-r,r)^{n-1} \times (-\delta, \delta)$ of $M$ around $x$ so that the leaves of $\mathcal{F}$ correspond to level sets of the last component.
After potentially shrinking $\delta$ the image of the flow line $\Phi(x,\blank) \colon (a,b) \to U$ through $x$ crosses every level set.
As $x \in \overline{M^\prime}$, there is some $y \in M^\prime \cap U$.
We consider the level set $(-r, r)^{n-1} \times \{t\}$ containing $\psi(y)$.
Within this set, $\psi(f^{-1}(\{f(y)\}) \cap U) = \psi(M^\prime \cap U) \cap ((-r, r)^{n-1} \times \{t\})$ is both open and closed, where the latter follows from properness of $f$.
Thus this whole level set is contained in $\psi(M^\prime \cap U)$.
In particular, a point in the flow line of $x$ is contained in $M^\prime$.
But then $x \in M^\prime$ by definition.
\end{proof}

\begin{proof}[Proof of \cref{Thm:ProdStr}]
We invoke \cref{Thm:FlowFol} for the vector field $X = -\bafrac{U}{|U|_g^2}$.
In order to make use of its full strength, we show that the foliation defined by the closed $1$-form $U^\flat$ extends (or actually coincides with) the foliation $(f^{-1}(t))_{t \in \R}$.
Since $U^\flat(X) = -1 \neq 0$, $X$ is transversal to the leaves of this foliation.

So let $Y$ be a vector in $Tf^{-1}(t)$ for some $t \in \R$.
We have to show that $g(U,Y)=0$.
We pull $Y$ back to $\mathcal{D}^\prime$ along $\Phi^\prime$ and extend this to a vector field on $\mathcal{D}^\prime$ in such a way that it is constant in the $\R$-direction.
Then the pushed-forward vector field $\widetilde{Y}$ on $M^\prime$ extends $Y$, is tangential to the foliation $(f^{-1}(t))_{t \in \R}$ everywhere and satisfies $[\widetilde{Y},X]=0$.

Since $g(U,\widetilde{Y})=0$ on $\del_+ M$ and every point in $M^\prime$ can be reached from there via a flow line of $X$ it suffices to show that $\del_X g(U, \widetilde{Y}) = 0$.
Since
\begin{align*}
[\widetilde{Y},X] &= -\left(\del_{\widetilde{Y}} \bafrac{1}{|U|_g^2}\right) U - \bafrac{1}{|U|_g^2} [\widetilde{Y},U] \\
	&= \bafrac{2 g(\nabla_{\widetilde{Y}} U, U)}{|U|_g^4} U - \bafrac{1}{|U|_g^2} [\widetilde{Y},U],
\end{align*}
the condition $[\widetilde{Y},X] = 0$ is equivalent to $2g(\nabla_{\widetilde{Y}} U, U)\, U = |U|_g^2 \, [\widetilde{Y},U]$, which implies
\begin{align*}
2g(\nabla_{\widetilde{Y}} U, U) &= g([\widetilde{Y},U],U).
\end{align*}
Note, moreover, that the condition $\upd U^\flat = 0$ is equivalent to
\begin{align*}
g(\nabla_A U, B) &= g(\nabla_B U, A)
\end{align*}
for any vectors $A$ and $B$.
Taking this together, we obtain the desired equation
\begin{align*}
-|U|_g^2 \del_X g(U, \widetilde{Y}) &= \del_U g(U, \widetilde{Y}) \\
	&= g(\nabla_U U, \widetilde{Y}) + g(U, \nabla_U \widetilde{Y}) \\
	&= g(\nabla_{\widetilde{Y}} U, U) + g(U, \nabla_{\widetilde{Y}} U) - g(U, [\widetilde{Y}, U]) \\
	&= 0.
\end{align*}

The previous theorem now establishes that $\Phi^\prime \colon \mathcal{D}^\prime \to M$ is a diffeomorphism and that $U$ is orthogonal to all the level sets of $f$.
Since $M$ and thus $\mathcal{D}^\prime$ is compact, there exists a maximal number $\ell \geq 0$ such that $\del_+ M \times \{\ell\}$ is contained in $\mathcal{D}^\prime$.
Maximality of $\ell$ implies that there exists some point $\Phi^\prime(x,\ell)$ that lies in $\del_- M$.
In particular, $\ell > 0$.
As $U$ is orthogonal to $\del_- M$, the connected component of $\del_- M$ that contains $\Phi^\prime(x,\ell)$ is a component of a leaf of the foliation defined by $U^\flat$, i.\,e.\ a component of the level set $f^{-1}(\ell)$.
But since $M$ and thus $\mathcal{D}^\prime \cong M$ is connected, also $\del_+ M$ and $f^{-1}(\ell) = \Phi^\prime(\del_+ M \times \{\ell\})$ are connected.
Thus $f^{-1}(\ell)$ is a component of $\del_- M$.
Since flow lines end when they reach $\del_- M$, this implies that $\mathcal{D}^\prime = \del_+ M \times [0,\ell]$.
\end{proof}

\end{document}